\documentclass[11pt,reqno]{amsart}

\usepackage[a4paper,hmargin=24mm,top=19mm,bottom=20mm]{geometry}
\usepackage{amsmath,amssymb,amsthm,mathtools}
\usepackage{enumitem}
\usepackage{microtype}
\usepackage{xcolor}
\usepackage{tikz-cd}
\usepackage[colorlinks=true,linkcolor=black,
            citecolor=black,urlcolor=black]{hyperref}

\newtheorem{theorem}{Theorem}[section]
\newtheorem{lemma}[theorem]{Lemma}
\newtheorem{corollary}[theorem]{Corollary}
\newtheorem{proposition}[theorem]{Proposition}
\theoremstyle{definition}
\newtheorem{definition}[theorem]{Definition}
\newtheorem{example}[theorem]{Example}
\theoremstyle{remark}
\newtheorem{remark}[theorem]{Remark}

\newcommand{\modu}{\operatorname{mod}}
\newcommand{\Fac}{\operatorname{Fac}}
\newcommand{\addc}{\operatorname{add}}
\newcommand{\Ext}{\operatorname{Ext}}
\newcommand{\Hom}{\operatorname{Hom}}
\newcommand{\Coker}{\operatorname{Coker}}
\newcommand{\stau}{\mathrm{s}\tau\text{-}\mathrm{tilt}}
\newcommand{\JassoRelativeCitation}{%
  \cite[Proposition~2.1(e)]{Jasso}}

\numberwithin{equation}{section}

\title[Support $\tau$-tilting modules over Morita contexts]
{Support $\tau$-tilting modules over Morita context algebras: A bilateral approximation approach}
\author[Yingying Zhang]{Yingying Zhang}
\address{Yingying Zhang, Department of Mathematics, Huzhou Normal University,
Huzhou 313000, Zhejiang Province, P.~R.~China}
\email{yyzhang@huznu.edu.cn}
\date{}

\hypersetup{
  pdftitle={Support tau-tilting modules over Morita context algebras: A bilateral approximation approach},
  pdfauthor={Yingying Zhang},
  pdfsubject={Support tau-tilting theory and Morita context algebras},
  pdfkeywords={support tau-tilting module, approximation gluing, Morita context algebra, torsion class}
}

\begin{document}

\begin{abstract}
Let $k$ be a field and let
$\Lambda=\left(\begin{smallmatrix}A&N\\M&B\end{smallmatrix}\right)_{\phi,\psi}$
be a finite-dimensional Morita context algebra.  We introduce a bilateral
approximation construction which glues support $\tau$-tilting modules over
$A$ and $B$ by alternately correcting the two corner components through
minimal approximations and pushouts.  When this process terminates, it yields
a support $\tau$-tilting $\Lambda$-module with the prescribed componentwise
torsion class.  The one-sided case recovers Zhang's triangular-matrix
construction, while the two-sided compatibility conditions give direct
corner induction and, for radical-valued connecting maps, are also necessary,
extending the Gao--Huang criterion.  Examples show that the bilateral
correction process can terminate even when neither one-sided compatibility
condition is satisfied, while in other examples the process never terminates.
\end{abstract}

\subjclass[2020]{16G10, 16E30, 16S50}
\keywords{support $\tau$-tilting module, approximation, Morita context
algebra, torsion class}

\maketitle
\enlargethispage{2pt}

\section{Introduction}

Support $\tau$-tilting theory identifies support $\tau$-tilting
modules with functorially finite torsion classes
\cite[Theorem~2.7]{AIR}; see also \cite[Theorem~2.11]{Jasso}.
Constructing a support $\tau$-tilting module with a prescribed torsion class
therefore amounts to proving functorial finiteness and identifying the
Ext-projective generator.  A $\tau$-rigid generator alone need not be
support $\tau$-tilting; see Example~\ref{ex:completion-term}.

For triangular matrix algebras, Gao and Huang characterized support
$\tau$-tilting modules obtained by direct corner induction
\cite{GaoHuang}.  Peng, Ma and Huang subsequently obtained related
construction and converse results \cite{PengMaHuang}.  Zhang identified the
corresponding glued module explicitly by replacing the direct cross term with
a minimal left approximation in a corner torsion class
\cite[Theorem~5.1]{Zhang}.  Related induction constructions for trivial
extensions were given by Li and Zhang \cite{LiZhang}.  For Morita rings with
zero connecting maps, Di--Tang--Yang studied weak silting modules
\cite{DiTangYang}, while Asefa--Xu characterized a directly induced silting
module under the stronger zero-product assumptions
$M\otimes_AN=0=N\otimes_BM$ \cite[Theorem~3.4]{AsefaXu}.

Zhang's theorem is obtained from a symmetric ladder of recollements.
Related categorical gluing results through good recollements were developed
by Wang, Luo, Liu, He and Wu \cite{WangLuoLiuHeWu}.  General Morita context
algebras need not carry the ladder used in the triangular case, which is why
the pushout corrections below are needed.

The purpose of this paper is to extend both Zhang's approximation construction
and the Gao--Huang direct-induction criterion from triangular matrix algebras
to Morita context algebras.  The two structure maps of a Morita context module
are coupled by the connecting maps $\phi$ and $\psi$; independently chosen
corner approximations therefore need not define a module.  We resolve this
problem by two pushout corrections and arrange the results in the chain
\[
 \text{bilateral correction}
 \ \Longrightarrow\ \text{one-sided correction}
 \ \Longrightarrow\ \text{direct corner induction}.
\]

Throughout, $k$ is a field, all algebras are finite-dimensional
$k$-algebras, and $\modu C$ denotes the category of finitely generated
right $C$-modules.  The algebras themselves are not assumed to be basic.
Subcategories are full and closed under isomorphisms.
The $\tau$-rigid and support $\tau$-tilting modules and
pairs are represented by basic objects;
the corner input modules denoted by $X$ and $Y$ are chosen in the same way.
Whenever a displayed direct sum represents one of these objects, repeated
isomorphic indecomposable summands are deleted without further comment.
Accordingly, in any statement $L\cong P(\mathcal T)$, a constructed module
$L$ is understood through its basic representative.  The convention does not
apply to objects occurring in approximation maps, exact sequences, or pushout
diagrams; their actual multiplicities are retained.

We briefly introduce the notation needed to state the main results.  Let
\[
 \Lambda=\begin{pmatrix}A&N\\M&B\end{pmatrix}_{\phi,\psi}
\]
be a finite-dimensional Morita context algebra, with diagonal idempotents
$e,f$, write $D=\Hom_k(-,k)$, and put
\[
 F_A=-\otimes_Ae\Lambda,
 \qquad F_B=-\otimes_Bf\Lambda.
\]
For $X\in\stau A$ and $Y\in\stau B$, set
\[
 \mathcal X=\Fac_AX,\qquad \mathcal Y=\Fac_BY,\qquad
 \mathcal T_{X,Y}=\{W\in\modu\Lambda\mid We\in\mathcal X,
                                      \ Wf\in\mathcal Y\}.
\]
For a functorially finite torsion class $\mathcal T$, write $P(\mathcal T)$
for its Ext-projective generator.
Section~3 constructs two alternating correction ladders starting from
$F_AX$ and $F_BY$.  Each step replaces one corner by a minimal approximation
and uses a pushout to modify the opposite corner so that the two Morita
context identities remain valid.  The construction is bilateral because the
two ladders alternately correct both corners.  A ladder terminates when both
components belong to their prescribed torsion classes.  Its \emph{minimal
target} is obtained by removing the zero component from the composite map.

Our first main result treats arbitrary connecting maps.

\begin{theorem}[See Theorem~\ref{thm:bilateral}]
\label{thm:intro-bilateral}
Assume that both correction ladders terminate.  Let $K$ be obtained from the
direct sum of the minimal targets by deleting repeated isomorphic summands.
Then $K$ is $\tau$-rigid, $\Fac_\Lambda K=\mathcal T_{X,Y}$, and for any
left $\addc K$-approximation $c:\Lambda\to K'$, one has
\[
 K\oplus\Coker c
 \cong P(\mathcal T_{X,Y})
\]
which is support $\tau$-tilting.
\end{theorem}

Thus $\mathcal T_{X,Y}$ is functorially finite under a concrete termination
hypothesis.  Example~\ref{ex:bilateral-finite} shows that both one-sided
containments may fail although both ladders terminate.  In the other
direction, Proposition~\ref{prop:obstruction} gives a broad obstruction to
unconditional termination, while Example~\ref{ex:zero-products} shows that
the obstruction persists even when both cross tensor products vanish.

If $X\otimes_AN\in\mathcal Y$, the $X$-ladder is already terminal and
the $Y$-ladder needs at most one correction.  In this case, a relative
presentation argument shows that the terminal objects form the full
Ext-projective generator.

\begin{theorem}[See Theorem~\ref{thm:main}]
\label{thm:intro-one-step}
Assume that $X\otimes_AN\in\Fac_BY$, and let
\[
 p:Y\otimes_BM\longrightarrow X_Y
\]
be a minimal left $\Fac_AX$-approximation.  If $K_p$ is the explicit pushout
correction of $F_BY$ associated with $p$, then
\[
 G_{\mathcal X}(X,Y)=F_AX\oplus K_p
 \cong P(\mathcal T_{X,Y}).
\]
In particular, $G_{\mathcal X}(X,Y)$ is support $\tau$-tilting and
$\Fac_\Lambda G_{\mathcal X}(X,Y)=\mathcal T_{X,Y}$.
\end{theorem}

When $N=0$, one has $F_AX=(X,0)$ and $K_p=(X_Y,Y)_p$, so
Theorem~\ref{thm:intro-one-step} is precisely Zhang's triangular construction
\cite[Theorem~5.1]{Zhang}.  The opposite one-sided version is given in
Theorem~\ref{thm:symmetric}.

If the second cross term also belongs to its prescribed torsion class, the
minimal approximation in Theorem~\ref{thm:intro-one-step} is the identity, so
the pushout correction is just $F_BY$.  This immediately gives the sufficient
part of direct corner induction for arbitrary connecting maps.  For the
converse we use the radical condition
\[
 \operatorname{Im}\phi\subseteq J(A),
 \qquad
 \operatorname{Im}\psi\subseteq J(B).
\]

\begin{theorem}[See Theorem~\ref{thm:direct-characterization}]
\label{thm:intro-direct}
Assume the radical condition, and let $X\in\modu A$ and $Y\in\modu B$.
Then
\[
 F_AX\oplus F_BY\text{ is support $\tau$-tilting}
\]
if and only if
\begin{align*}
 X&\in\stau A, &Y&\in\stau B,\\
 Y\otimes_BM&\in\Fac_AX, &X\otimes_AN&\in\Fac_BY.
\end{align*}
In this case, if $(X,P_X)$ and $(Y,P_Y)$ are the corresponding support $\tau$-tilting
pairs, then the projective complement is $F_AP_X\oplus F_BP_Y$, and
\[
 \Fac_\Lambda(F_AX\oplus F_BY)=\mathcal T_{X,Y}.
\]
\end{theorem}

The sufficient direction is the zero-correction specialization and holds
for arbitrary connecting maps, as does the $\tau$-rigidity criterion in
Lemma~\ref{lem:tau-rigid-direct}.  The radical condition is used in the
summand count for the converse.  For $N=0$, this extends the Gao--Huang
criterion from algebraically closed fields to arbitrary fields.

Section~2 recalls the required background.  Section~3 constructs the corner
corrections, proves the bilateral theorem, and gives obstructions to
unconditional termination.  Section~4 obtains the one-sided theorem and
Zhang's construction by specialization.  Section~5 derives direct induction
and proves its converse under radical pairings.

\section{Preliminaries}

\subsection{Morita context algebras and their modules}

Let $A$ and $B$ be finite-dimensional $k$-algebras,
let ${}_AN_B$ and ${}_BM_A$ be finite-dimensional bimodules, and let
\[
 \phi:N\otimes_BM\longrightarrow A,
 \qquad
 \psi:M\otimes_AN\longrightarrow B
\]
be compatible bimodule homomorphisms.  Thus
\begin{align}
 \phi(n\otimes m)n'&=n\psi(m\otimes n'),
 &m\phi(n\otimes m')&=\psi(m\otimes n)m'
 \label{eq:context-identities}
\end{align}
for $n,n'\in N$ and $m,m'\in M$.  They determine the Morita context algebra
\[
 \Lambda=
 \begin{pmatrix}A&N\\M&B\end{pmatrix}_{\phi,\psi}.
\]
Its multiplication is
\[
 \begin{pmatrix}a&n\\m&b\end{pmatrix}
 \begin{pmatrix}a'&n'\\m'&b'\end{pmatrix}
 =
 \begin{pmatrix}
  aa'+\phi(n\otimes m')&an'+nb'\\
  ma'+bm'&\psi(m\otimes n')+bb'
 \end{pmatrix}.
\]
Thus, inside $\Lambda$,
\[
 nm=\phi(n\otimes m),\qquad mn=\psi(m\otimes n),
 \qquad
 \operatorname{Im}\phi=NM,\quad \operatorname{Im}\psi=MN;
\]
see also \cite[Section~2]{GreenPsaroudakis}.  The underlying $k$-vector
space is $A\oplus N\oplus M\oplus B$, so $\Lambda$ is finite-dimensional.
Since a finitely generated module over a finite-dimensional $k$-algebra is
finite-dimensional over $k$, it would equivalently suffice to assume that
$M$ and $N$ are finitely generated on either one of their indicated sides.
Throughout, $D=\Hom_k(-,k)$ denotes the $k$-duality.
Write $e$ and $f$ for the two diagonal idempotents.  We use right modules.
A right $\Lambda$-module is a quadruple
\[
 W=(U,V,u,v),
\]
where $U\in\modu A$, $V\in\modu B$, and
\[
 u:U\otimes_AN\longrightarrow V,
 \qquad
 v:V\otimes_BM\longrightarrow U
\]
satisfy
\begin{align}
 v(u(x\otimes n)\otimes m)&=x\phi(n\otimes m),
 \label{eq:module-one}\\
 u(v(y\otimes m)\otimes n)&=y\psi(m\otimes n).
 \label{eq:module-two}
\end{align}
A morphism $(c,d):(U,V,u,v)\to(U',V',u',v')$ consists of an
$A$-linear map $c:U\to U'$ and a $B$-linear map $d:V\to V'$ satisfying
\[
 du=u'(c\otimes1_N),\qquad cv=v'(d\otimes1_M).
\]
Exactness of a sequence of Morita context modules is detected componentwise.

The corner induction functors are
\[
 F_A=-\otimes_Ae\Lambda,
 \qquad
 F_B=-\otimes_Bf\Lambda.
\]
By \cite[Proposition~2.4(ii)]{GreenPsaroudakis}, we have the adjoint pairs
\[
 F_A\dashv(-)e,
 \qquad
 F_B\dashv(-)f.
\]
By \cite[Proposition~2.4(i)]{GreenPsaroudakis}, both $F_A$ and $F_B$ are
fully faithful.  By \cite[Proposition~2.4(iii)]{GreenPsaroudakis}, both
corner restriction functors are exact.  Explicitly,
\begin{align*}
 F_AU&=(U,U\otimes_AN,1_{U\otimes N},1_U\otimes\phi),\\
 F_BV&=(V\otimes_BM,V,1_V\otimes\psi,1_{V\otimes M}).
\end{align*}

\subsection{Support \texorpdfstring{$\tau$}{tau}-tilting modules and
approximation tools}

Let $C$ be a finite-dimensional algebra and $\tau_C$ its
Auslander--Reiten translation.  For $T\in\modu C$, write $|T|$ for the
number of isomorphism classes of its indecomposable direct summands; in
particular, $|C|$ is the number of simple $C$-modules up to isomorphism.

\begin{definition}\label{def:tau-tilting}
Let $T\in\modu C$ and let $P$ be projective.
\begin{enumerate}[label=\textup{(\arabic*)},leftmargin=2.2em,nosep]
 \item The module $T$ is \emph{$\tau$-rigid} if
 $\Hom_C(T,\tau_CT)=0$.  It is \emph{$\tau$-tilting} if, in addition,
 $|T|=|C|$.
 \item The pair $(T,P)$ is a \emph{$\tau$-rigid pair} if $T$ is
 $\tau$-rigid and $\Hom_C(P,T)=0$.  It is a \emph{support
 $\tau$-tilting pair} if, in addition, $|T|+|P|=|C|$.
 \item The module $T$ is \emph{support $\tau$-tilting} if it is the first
 component of such a pair.  Write $\stau C$ for the set of isomorphism
 classes of support $\tau$-tilting $C$-modules.
\end{enumerate}
\end{definition}

We call $C$ \emph{$\tau$-tilting finite} if $\stau C$ is finite; this is
equivalent to finiteness of the set of $\tau$-tilting modules
\cite[Proposition~3.9]{DIJ}.

\begin{definition}\label{def:torsion-notation}
Let $T\in\modu C$ and let $\mathcal T$ be a full subcategory of $\modu C$.
\begin{enumerate}[label=\textup{(\arabic*)},leftmargin=2.2em,nosep]
 \item The subcategory $\addc T$ consists of the direct summands of finite
 direct sums of copies of $T$, and $\Fac_CT$ consists of the factor modules
 of objects of $\addc T$.
 \item The subcategory $\mathcal T$ is a \emph{torsion class} if it is
 nonempty and closed under factor modules and extensions.
 \item An object $E\in\mathcal T$ is \emph{Ext-projective in $\mathcal T$}
 if $\Ext_C^1(E,T')=0$ for every $T'\in\mathcal T$.  We write
 $\mathcal P(\mathcal T)$ for the full subcategory of such objects.  If it
 has a finite additive generator, then $P(\mathcal T)$ denotes the direct
 sum of one representative from each isomorphism class of its
 indecomposable objects.
\end{enumerate}
\end{definition}

We record the support $\tau$-tilting--torsion correspondence used below.
This also fixes the precise meaning of the Ext-projective generator appearing
in the construction.

\begin{theorem}
\label{thm:AIR-torsion}
Let $T$ be a finite-dimensional right $C$-module.
\begin{enumerate}[label=\textup{(\arabic*)},leftmargin=2.2em]
 \item \textup{(\cite[Proposition~1.2]{AIR})} $T$ is $\tau$-rigid if and only if
$\Ext_C^1(T,\Fac_CT)=0$.  If these conditions hold, then $\Fac_CT$ is a
functorially finite torsion class and
\[
 T\in\addc P(\Fac_CT).
\]
 \item \textup{(\cite[Theorem~2.7]{AIR})} The assignment
$T\mapsto\Fac_CT$ is a bijection from
$\stau C$ to the functorially finite torsion classes in $\modu C$; its
inverse sends a functorially finite torsion class $\mathcal U$ to
$P(\mathcal U)$.
 \item \textup{(\cite[Corollary~2.13(c)]{AIR})} If $(T,P)$ is a support
$\tau$-tilting pair, then
 \begin{equation}\label{eq:Fac-orthogonality}
  \Fac_CT
  =\{L\in\modu C\mid
       \Hom_C(L,\tau_CT)=0=\Hom_C(P,L)\}.
 \end{equation}
\end{enumerate}
\end{theorem}

We shall also use the following presentation, approximation, and numerical
criteria.

\begin{proposition}
\label{prop:AIR-criteria}
Let $T,L\in\modu C$, let
$\sigma_T:P_1^T\to P_0^T$ be a minimal projective presentation of $T$,
and let $P$ be projective.
\begin{enumerate}[label=\textup{(\arabic*)},leftmargin=2.2em]
 \item \textup{(\cite[Proposition~2.4(b)]{AIR})} The map
$\Hom_C(\sigma_T,L)$ is surjective if and only if
 $\Hom_C(L,\tau_CT)=0$.
 \item \textup{(\cite[Propositions~1.3 and 2.3(a)]{AIR})} If $(T,P)$ is a
$\tau$-rigid pair, then
 $|T|+|P|\leq |C|$.
 \item \textup{(\cite[Proposition~2.3(b)]{AIR})} If $T$ is support
$\tau$-tilting, then the projective module
 $P$ for which $(T,P)$ is a support $\tau$-tilting pair is unique up to
 isomorphism.
 \item \textup{(\cite[Proposition~2.14]{Jasso})} A $\tau$-rigid module $T$
is support $\tau$-tilting if and only if there is an exact sequence
\[
 C\xrightarrow{f}T_0\longrightarrow T_1\longrightarrow0
\]
with $T_0,T_1\in\addc T$ such that $f$ is a left
$\addc T$-approximation.
\end{enumerate}
\end{proposition}

We denote the projective module in
Proposition~\ref{prop:AIR-criteria}\textup{(3)} by $P_T$.

\begin{definition}\label{def:approximations}
Let $\mathcal C$ be a full subcategory of $\modu C$.
\begin{enumerate}[label=\textup{(\arabic*)},leftmargin=2.2em,nosep]
 \item A morphism $a:L\to C_L$ with $C_L\in\mathcal C$ is a \emph{left
 $\mathcal C$-approximation} if every morphism from $L$ to an object of
 $\mathcal C$ factors through $a$.  Right approximations are defined
 dually.
 \item A left approximation $a:L\to C_L$ is \emph{left minimal} if every
 endomorphism $s$ of $C_L$ satisfying $sa=a$ is invertible.  Right
 minimality is defined dually.
 \item The subcategory $\mathcal C$ is \emph{covariantly finite} if every
 module admits a left $\mathcal C$-approximation, \emph{contravariantly
 finite} if every module admits a right $\mathcal C$-approximation, and
 \emph{functorially finite} if it has both properties.
\end{enumerate}
\end{definition}

In the present Hom-finite Krull--Schmidt setting, an approximation into a
subcategory closed under direct summands decomposes into a minimal
approximation and a zero component.  Minimal approximations are unique up
to isomorphism.

The next lemma records the standard relative presentation property of a
functorially finite torsion class.  We include a proof that applies to any
right $\addc T$-approximation.

\begin{lemma}[cf.~\JassoRelativeCitation]
\label{lem:relative-covers}
Let $(T,P)$ be a support $\tau$-tilting pair over a finite-dimensional
algebra $C$ and let $U\in\Fac_CT$.  Every right $\addc T$-approximation
$g:T_0\to U$ fits into an exact sequence
\[
 0\longrightarrow U'\longrightarrow T_0\xrightarrow{g}U\longrightarrow0
\]
with $U'\in\Fac_CT$.  In particular, every object of $\Fac_CT$ admits
such a sequence with $T_0\in\addc T$.
\end{lemma}

\begin{proof}
Since $U\in\Fac_CT$, an epimorphism from an object of $\addc T$ to $U$
factors through $g$, so $g$ is epic.  Put $U'=\ker g$.
By \cite[Lemma~2.6]{AIR}, one has $\Hom_C(U',\tau_CT)=0$.
The inclusion $U'\hookrightarrow T_0$ and $\Hom_C(P,T_0)=0$ give
$\Hom_C(P,U')=0$.  Formula~\eqref{eq:Fac-orthogonality} now gives
$U'\in\Fac_CT$.  For the last assertion, take the evaluation map
$T^{\oplus s}\to U$ determined by a $k$-basis of $\Hom_C(T,U)$; it is a
right $\addc T$-approximation.
\end{proof}

Section~3 also records a relative Wakamatsu lemma
(Lemma~\ref{lem:Wakamatsu}) and a module form of co-Bongartz completion
(Proposition~\ref{prop:module-coBongartz}), followed by a small example
illustrating the completion term.  Remark~\ref{rem:completion-map}
discusses convenient minimal and canonical choices of the completion map.

\subsection{The componentwise torsion class}

Let $X\in\stau A$ and $Y\in\stau B$, and put
\[
 \mathcal X=\Fac_AX,
 \qquad
 \mathcal Y=\Fac_BY.
\]
We consider the componentwise class
\[
 \mathcal T_{X,Y}
 =\{(U,V,u,v)\in\modu\Lambda\mid U\in\mathcal X,
                                      \ V\in\mathcal Y\}.
\]
It is a torsion class because quotients and extensions are computed
componentwise.

\section{Bilateral approximation gluing}

\subsection{Elementary corner corrections}

Put
\[
 \mathcal C_{\mathcal X}
 =\{W\in\modu\Lambda\mid We\in\mathcal X\},
 \qquad
 \mathcal C_{\mathcal Y}
 =\{W\in\modu\Lambda\mid Wf\in\mathcal Y\}.
\]
Both are torsion classes and
$\mathcal T_{X,Y}=\mathcal C_{\mathcal X}\cap\mathcal C_{\mathcal Y}$.
The classes $\mathcal X$ and $\mathcal Y$ are functorially finite, so the
minimal left approximations used below exist.  They, and hence the resulting
corrections, are unique up to isomorphism.

We next define two elementary correction operations.  Let
$W=(U,V,u,v)$.  For the $\mathcal X$-correction, choose a minimal left
$\mathcal X$-approximation $a_W:U\to U^{\mathcal X}$ and form the pushout
in $\modu B$
\begin{equation}\label{eq:RX-pushout}
\begin{tikzcd}[column sep=4.2em,row sep=2.8em]
 U\otimes_AN \arrow[r,"u"]\arrow[d,"a_W\otimes1_N"']
 &V\arrow[d,"\iota_{\mathcal X}^W"]\\
 U^{\mathcal X}\otimes_AN
 \arrow[r,"\mu_{\mathcal X}^W"']&V^{\mathcal X} .
\end{tikzcd}
\end{equation}
Define two $A$-module homomorphisms
\begin{align*}
 r_V^W&:=a_Wv:V\otimes_BM\longrightarrow U^{\mathcal X},\\
 r_N^W&:=1_{U^{\mathcal X}}\otimes\phi:
 U^{\mathcal X}\otimes_AN\otimes_BM\longrightarrow U^{\mathcal X}.
\end{align*}
Here and below we use the canonical identification
$U^{\mathcal X}\otimes_AA\cong U^{\mathcal X}$; in particular,
$r_N^W(x_{\mathcal X}\otimes n\otimes m)
=x_{\mathcal X}\phi(n\otimes m)$.
The first module identity for $W$ and the $A$-linearity of $a_W$ give
\[
 r_V^W(u\otimes1_M)
 =r_N^W(a_W\otimes1_N\otimes1_M).
\]
Since $-\otimes_BM$ preserves pushouts, tensoring
\eqref{eq:RX-pushout} with $M$ gives a pushout square in $\modu A$.
Therefore this equality and the universal property of that pushout yield a
unique map
$v_{\mathcal X}^W:V^{\mathcal X}\otimes_BM\to U^{\mathcal X}$ satisfying
\begin{align*}
 v_{\mathcal X}^W(\iota_{\mathcal X}^W\otimes1_M)
 &=r_V^W=a_Wv,\\
 v_{\mathcal X}^W(\mu_{\mathcal X}^W\otimes1_M)
 &=r_N^W.
\end{align*}
Set
\[
 R_{\mathcal X}(W)
 =(U^{\mathcal X},V^{\mathcal X},
   \mu_{\mathcal X}^W,v_{\mathcal X}^W),
 \qquad
 \rho_{\mathcal X}^W=(a_W,\iota_{\mathcal X}^W):
 W\longrightarrow R_{\mathcal X}(W).
\]

For the $\mathcal Y$-correction, choose a minimal left
$\mathcal Y$-approximation $b_W:V\to V^{\mathcal Y}$ and form the pushout
in $\modu A$
\[
\begin{tikzcd}[column sep=4.2em,row sep=2.8em]
 V\otimes_BM \arrow[r,"v"]\arrow[d,"b_W\otimes1_M"']
 &U\arrow[d,"\kappa_{\mathcal Y}^W"]\\
 V^{\mathcal Y}\otimes_BM
 \arrow[r,"\nu_{\mathcal Y}^W"']&U^{\mathcal Y} .
\end{tikzcd}
\]
Define two $B$-module homomorphisms
\begin{align*}
 s_U^W&:=b_Wu:U\otimes_AN\longrightarrow V^{\mathcal Y},\\
 s_M^W&:=1_{V^{\mathcal Y}}\otimes\psi:
 V^{\mathcal Y}\otimes_BM\otimes_AN\longrightarrow V^{\mathcal Y}.
\end{align*}
Thus
$s_M^W(y_{\mathcal Y}\otimes m\otimes n)
=y_{\mathcal Y}\psi(m\otimes n)$.
The second module identity for $W$ and the $B$-linearity of $b_W$ give
\[
 s_U^W(v\otimes1_N)
 =s_M^W(b_W\otimes1_M\otimes1_N).
\]
Since $-\otimes_AN$ preserves pushouts, tensoring
this pushout square with $N$ gives a pushout square in $\modu B$.
Therefore this equality and the universal property yield a unique map
$u_{\mathcal Y}^W:U^{\mathcal Y}\otimes_AN\to V^{\mathcal Y}$ satisfying
\begin{align*}
 u_{\mathcal Y}^W(\kappa_{\mathcal Y}^W\otimes1_N)
 &=s_U^W=b_Wu,\\
 u_{\mathcal Y}^W(\nu_{\mathcal Y}^W\otimes1_N)
 &=s_M^W.
\end{align*}
Set
\[
 R_{\mathcal Y}(W)
 =(U^{\mathcal Y},V^{\mathcal Y},
   u_{\mathcal Y}^W,\nu_{\mathcal Y}^W),
 \qquad
 \rho_{\mathcal Y}^W=(\kappa_{\mathcal Y}^W,b_W):
 W\longrightarrow R_{\mathcal Y}(W).
\]

The preceding universal-property constructions are well defined in the
Morita context module category.

\begin{proposition}[Well-definedness of the corrections]
\label{prop:well-defined-corrections}
For every $W\in\modu\Lambda$, the quadruples
$R_{\mathcal X}(W)$ and $R_{\mathcal Y}(W)$ are $\Lambda$-modules, and
\[
 \rho_{\mathcal X}^W:W\longrightarrow R_{\mathcal X}(W),
 \qquad
 \rho_{\mathcal Y}^W:W\longrightarrow R_{\mathcal Y}(W)
\]
are morphisms of $\Lambda$-modules.  Moreover,
\[
 R_{\mathcal X}(W)\in\mathcal C_{\mathcal X},
 \qquad
 R_{\mathcal Y}(W)\in\mathcal C_{\mathcal Y}.
\]
\end{proposition}

\begin{proof}
We prove the assertions for $R_{\mathcal X}(W)$; the other case is obtained
by interchanging the two corners.  Suppress the superscript $W$ on
$\iota_{\mathcal X}^W$, $\mu_{\mathcal X}^W$, and
$v_{\mathcal X}^W$.  The first module identity is the second defining
identity for $v_{\mathcal X}^W$.  For the second identity, the images of
$\iota_{\mathcal X}$ and $\mu_{\mathcal X}$ generate
$V^{\mathcal X}$, and
\begin{align*}
 \mu_{\mathcal X}\bigl(
   v_{\mathcal X}(\iota_{\mathcal X}(y)\otimes m)\otimes n\bigr)
 &=\iota_{\mathcal X}\bigl(u(v(y\otimes m)\otimes n)\bigr)
  =\iota_{\mathcal X}\bigl(y\psi(m\otimes n)\bigr),\\
 \mu_{\mathcal X}\bigl(
   v_{\mathcal X}(\mu_{\mathcal X}(x_{\mathcal X}\otimes n')\otimes m)
   \otimes n\bigr)
 &=\mu_{\mathcal X}\bigl(
   x_{\mathcal X}\phi(n'\otimes m)\otimes n\bigr)\\
 &=\mu_{\mathcal X}(x_{\mathcal X}\otimes n')\psi(m\otimes n).
\end{align*}
The first line uses the first defining identity for $v_{\mathcal X}^W$,
the pushout relation, and \eqref{eq:module-two} for $W$.  The second line
uses its second defining identity and then
\eqref{eq:context-identities}.  Hence
$R_{\mathcal X}(W)$ satisfies
\eqref{eq:module-one}--\eqref{eq:module-two}.

The two identities
\[
 \iota_{\mathcal X}u
 =\mu_{\mathcal X}(a_W\otimes1_N),
 \qquad
 a_Wv=v_{\mathcal X}(\iota_{\mathcal X}\otimes1_M)
\]
are respectively the pushout relation in \eqref{eq:RX-pushout} and the
first defining identity for $v_{\mathcal X}^W$.  They are precisely the two
morphism identities for
$\rho_{\mathcal X}^W=(a_W,\iota_{\mathcal X})$.  Finally,
$U^{\mathcal X}\in\mathcal X$, so
$R_{\mathcal X}(W)\in\mathcal C_{\mathcal X}$.
\end{proof}

The superscript on $U^{\mathcal X},V^{\mathcal X}$ or
$U^{\mathcal Y},V^{\mathcal Y}$ records which correction was applied; it
does not indicate a change of corner.

\subsection{Finite correction ladders}

Starting from the two induced modules, define the alternating sequences
schematically by
\begin{align*}
 E_X^0=F_AX
 &\xrightarrow{\rho_{\mathcal Y}^{E_X^0}}E_X^1
 \xrightarrow{\rho_{\mathcal X}^{E_X^1}}E_X^2
 \xrightarrow{\rho_{\mathcal Y}^{E_X^2}}E_X^3\longrightarrow\cdots,\\
 E_Y^0=F_BY
 &\xrightarrow{\rho_{\mathcal X}^{E_Y^0}}E_Y^1
 \xrightarrow{\rho_{\mathcal Y}^{E_Y^1}}E_Y^2
 \xrightarrow{\rho_{\mathcal X}^{E_Y^2}}E_Y^3\longrightarrow\cdots,
\end{align*}
where $E_X^{i+1}$ and $E_Y^{i+1}$ are the targets of the displayed
correction morphisms.  At every step we use a minimal corner approximation.
These are the $X$- and $Y$-correction ladders.  A ladder terminates at its
first term belonging to $\mathcal T_{X,Y}$.
The chosen minimal approximations, and hence the correction terms, are
determined only up to isomorphism; no functor on all morphisms is being
asserted.

The next lemma records the approximation property of the corrections
constructed in Proposition~\ref{prop:well-defined-corrections}.

\begin{lemma}[Approximation property]\label{lem:elementary-correction}
For every $W\in\modu\Lambda$, the morphism $\rho_{\mathcal X}^W$ is a
minimal left $\mathcal C_{\mathcal X}$-approximation.  Symmetrically,
$\rho_{\mathcal Y}^W$ is a minimal left
$\mathcal C_{\mathcal Y}$-approximation.
\end{lemma}

\begin{proof}
We prove the assertion for $\rho_{\mathcal X}^W$ and suppress the
superscript $W$ on $\iota_{\mathcal X}^W$,
$\mu_{\mathcal X}^W$, and $v_{\mathcal X}^W$.
Let $Z=(U_Z,V_Z,u_Z,v_Z)\in\mathcal C_{\mathcal X}$ and let
$h=(c,d):W\to Z$.  Since $a_W$ is a left $\mathcal X$-approximation,
write $c=qa_W$, where $q:U^{\mathcal X}\to U_Z$.  To apply the pushout
property, we verify compatibility on its common source $U\otimes_AN$.
Since $h$ is a module morphism, its first morphism identity is
$du=u_Z(c\otimes1_N)$.  Together with $c=qa_W$, this gives
\[
 du=u_Z(c\otimes1_N)=u_Z((qa_W)\otimes1_N)
   =u_Z(q\otimes1_N)(a_W\otimes1_N).
\]
Thus the maps $d:V\to V_Z$ and
$u_Z(q\otimes1_N):U^{\mathcal X}\otimes_AN\to V_Z$ agree after
precomposition with $u$ and $a_W\otimes1_N$, respectively.  The universal
property of \eqref{eq:RX-pushout} consequently gives a unique map
$r:V^{\mathcal X}\to V_Z$ such that
\[
 r\iota_{\mathcal X}=d,
 \qquad r\mu_{\mathcal X}=u_Z(q\otimes1_N).
\]
The second equation is precisely the first morphism identity for
$(q,r):R_{\mathcal X}(W)\to Z$.  To verify the second morphism identity,
it is enough to check the two sets of generators of $V^{\mathcal X}$:
\begin{align*}
 v_Z(r\iota_{\mathcal X}(y)\otimes m)
 &=v_Z(d(y)\otimes m)=c(v(y\otimes m))\\
 &=qv_{\mathcal X}(\iota_{\mathcal X}(y)\otimes m),\\
 v_Z(r\mu_{\mathcal X}(x_{\mathcal X}\otimes n)\otimes m)
 &=v_Z(u_Z(q(x_{\mathcal X})\otimes n)\otimes m)\\
 &=q(x_{\mathcal X})\phi(n\otimes m)\\
 &=qv_{\mathcal X}(\mu_{\mathcal X}(x_{\mathcal X}\otimes n)\otimes m).
\end{align*}
The equality involving $c$ is the second morphism identity for $h$, while
the equality involving $\phi$ is the first module identity for $Z$.
Hence $(q,r)$ is a module morphism and
$h=(q,r)\rho_{\mathcal X}^W$.

If an endomorphism $s=(s_U,s_V)$ of $R_{\mathcal X}(W)$ satisfies
$s\rho_{\mathcal X}^W=\rho_{\mathcal X}^W$, then $s_Ua_W=a_W$, so $s_U$ is invertible by the
minimality of $a_W$.  Moreover,
$s_V\iota_{\mathcal X}=\iota_{\mathcal X}$ and
$s_V\mu_{\mathcal X}=\mu_{\mathcal X}(s_U\otimes1_N)$.  Since the images
of $\iota_{\mathcal X}$ and $\mu_{\mathcal X}$ generate
$V^{\mathcal X}$, the map $s_V$ is surjective and hence invertible.  Thus
$\rho_{\mathcal X}^W$ is left minimal.  The proof for
$R_{\mathcal Y}$ is obtained by interchanging the two corners.
\end{proof}

\begin{lemma}\label{lem:corner-Ext-projective}
The induced modules satisfy
\[
 \Ext^1_\Lambda(F_AX,\mathcal T_{X,Y})=0
 =\Ext^1_\Lambda(F_BY,\mathcal T_{X,Y}).
\]
\end{lemma}

\begin{proof}
Let $W\in\mathcal T_{X,Y}$ and take an extension
$0\to W\to E\xrightarrow{\pi}F_AX\to0$.  Applying the exact corner
restriction $(-)e$, and identifying $(F_AX)e$ with $X$, gives
$0\to We\to Ee\xrightarrow{\pi_e}X\to0$.  Since
$We\in\mathcal X=\Fac_AX$ and $X$ is $\tau$-rigid,
$\Ext_A^1(X,We)=0$ by
Theorem~\ref{thm:AIR-torsion}\textup{(1)}.  Hence $\pi_e$
has an $A$-linear section $s_e:X\to Ee$.

Write
$\Theta_E:\Hom_\Lambda(F_AX,E)\xrightarrow{\sim}\Hom_A(X,Ee)$ for the
adjunction bijection, and let $s=\Theta_E^{-1}(s_e)$.  Naturality gives
\[
 \Theta_{F_AX}(\pi s)=\pi_e\Theta_E(s)=\pi_es_e=1_X.
\]
The adjunction unit $X\to(F_AX)e$ is an isomorphism, and under our
identification $\Theta_{F_AX}(1_{F_AX})=1_X$.  Since $\Theta_{F_AX}$ is
injective, $\pi s=1_{F_AX}$.  Thus the original extension splits.  This
proves $\Ext_\Lambda^1(F_AX,\mathcal T_{X,Y})=0$.  The proof of the other
equality is the same, using $f,B,Y$, and $F_B$.
\end{proof}

\begin{lemma}[Relative Wakamatsu lemma]\label{lem:Wakamatsu}
Let $\Gamma$ be a finite-dimensional algebra and $\mathcal T$ an
extension-closed subcategory of $\modu\Gamma$.  Suppose that
$E\in\modu\Gamma$ satisfies
\[
 \Ext^1_\Gamma(E,\mathcal T)=0.
\]
If $a:E\to K$ is a minimal left $\mathcal T$-approximation and
$K\in\mathcal T$, then
\[
 \Ext^1_\Gamma(K,\mathcal T)=0.
\]
\end{lemma}

\begin{proof}
Fix $T\in\mathcal T$ and an extension
\[
 0\longrightarrow T\longrightarrow C\xrightarrow{s}K\longrightarrow0.
\]
Since $T,K\in\mathcal T$ and $\mathcal T$ is extension closed, one has
$C\in\mathcal T$.  Applying $\Hom_\Gamma(E,-)$ and using
$\Ext^1_\Gamma(E,T)=0$ shows that $a$ lifts to a map $h:E\to C$ with
$sh=a$.  Since $a$ is a left $\mathcal T$-approximation, write $h=ta$ for
some $t:K\to C$.  Consequently,
\[
 (st)a=a.
\]
The left minimality of $a$ implies that $st$ is invertible.  Hence
$t(st)^{-1}$ is a section of $s$, so the extension splits.  Therefore
$\Ext^1_\Gamma(K,T)=0$.
\end{proof}

At module level, Cao calls $P(\Fac_\Gamma K)$ the co-Bongartz completion of
a $\tau$-rigid module $K$ \cite[Definition~2.10]{Cao}; compare the classical
Bongartz completion \cite{Bongartz}.  The following result can also be obtained
from the work on completions of two-term silting complexes in
\cite[Section~5]{DerksenFei}, \cite[Proposition~2.16]{Aihara},
\cite[Lemma~4.2]{IJY}, and \cite[Lemma~2.2(b)]{Kimura}, together with the
two-term silting--support $\tau$-tilting correspondence
\cite[Theorem~3.2]{AIR}.  For completeness, we give a direct module-theoretic
proof.

\begin{proposition}[Module-level co-Bongartz formula]
\label{prop:module-coBongartz}
Let $\Gamma$ be a finite-dimensional $k$-algebra and let $K$ be a
$\tau$-rigid right $\Gamma$-module.  If
$c:\Gamma\longrightarrow K'$ is a left $\addc K$-approximation, then
\[
 G_c=K\oplus\Coker c
\]
is a support $\tau$-tilting $\Gamma$-module and
\[
 \Fac_\Gamma G_c=\Fac_\Gamma K,
 \qquad
 G_c\cong P(\Fac_\Gamma K).
\]
\end{proposition}

\begin{proof}
Put $\mathcal T=\Fac_\Gamma K$.  Since $K$ is $\tau$-rigid,
Theorem~\ref{thm:AIR-torsion}\textup{(1)} gives
$\Ext_\Gamma^1(K,\mathcal T)=0$, and $\mathcal T$ is a functorially finite
torsion class.

We first show that $c$ is a left $\mathcal T$-approximation.  Let
$Z\in\mathcal T$ and $\xi:\Gamma\to Z$.  Choose an epimorphism
$p:K_0\twoheadrightarrow Z$ with $K_0\in\addc K$.  Since $\Gamma$ is
projective, $\xi$ lifts to $\widetilde\xi:\Gamma\to K_0$ with
$p\widetilde\xi=\xi$.  As $c$ is a left $\addc K$-approximation,
$\widetilde\xi=hc$ for some $h:K'\to K_0$; hence $\xi=phc$.

Now $\Coker c\in\mathcal T$, since it is a factor module of
$K'\in\addc K$.  Put $I=\operatorname{Im}c$ and write $c=i\pi$, where
$\pi:\Gamma\twoheadrightarrow I$ and $i:I\hookrightarrow K'$.  Fix
$L\in\mathcal T$ and let $a:I\to L$.  Since $c$ is a left
$\mathcal T$-approximation, the map $a\pi:\Gamma\to L$ factors through
$c$; thus $a\pi=bc=bi\pi$ for some $b:K'\to L$.  As $\pi$ is epic,
$a=bi$, so
\[
 \Hom_\Gamma(K',L)\longrightarrow\Hom_\Gamma(I,L)
\]
is surjective.  Applying $\Hom_\Gamma(-,L)$ to
\[
 0\longrightarrow I\longrightarrow K'\longrightarrow
 \Coker c\longrightarrow0
\]
and using $\Ext_\Gamma^1(K',L)=0$ gives
$\Ext_\Gamma^1(\Coker c,L)=0$.  Hence $\Coker c$ is Ext-projective in
$\mathcal T$.

Thus $\Fac_\Gamma G_c=\Fac_\Gamma K=\mathcal T$ and
$\Ext_\Gamma^1(G_c,\mathcal T)=0$, so
Theorem~\ref{thm:AIR-torsion}\textup{(1)} shows that $G_c$ is
$\tau$-rigid.  Moreover,
\[
 \Gamma\xrightarrow{c}K'\longrightarrow\Coker c\longrightarrow0
\]
is exact with $K',\Coker c\in\addc G_c$.  Since $c$ is a left
$\mathcal T$-approximation and $\addc G_c\subseteq\mathcal T$, it is a
left $\addc G_c$-approximation.  Proposition~\ref{prop:AIR-criteria}
\textup{(4)} therefore shows that $G_c$ is support $\tau$-tilting.
Finally, Theorem~\ref{thm:AIR-torsion}\textup{(2)} gives
\[
 \Fac_\Gamma G_c=\Fac_\Gamma K,
 \qquad
 G_c\cong P(\Fac_\Gamma K).
\]
\end{proof}

\begin{example}[A nontrivial co-Bongartz completion]
\label{ex:completion-term}
Let $\Gamma=k(1\xrightarrow{\alpha}2)$, and write $\mathbf 1$ and
$\mathbf 2$ for the simple right $\Gamma$-modules at the corresponding
vertices.  Put
\[
 K=\genfrac{}{}{0pt}{}{\mathbf 1}{\mathbf 2}=e_1\Gamma.
\]
Then
\[
 \Fac_\Gamma K
 =\addc\left(
   \genfrac{}{}{0pt}{}{\mathbf 1}{\mathbf 2}\oplus\mathbf 1\right),
 \qquad
 P(\Fac_\Gamma K)
 =\genfrac{}{}{0pt}{}{\mathbf 1}{\mathbf 2}\oplus\mathbf 1.
\]
Thus $K$ is $\tau$-rigid and generates $\Fac_\Gamma K$, but it is not
support $\tau$-tilting.  Since
\[
 \Gamma=\genfrac{}{}{0pt}{}{\mathbf 1}{\mathbf 2}\oplus\mathbf 2,
\]
a minimal left $\addc K$-approximation is
\[
 c:
 \genfrac{}{}{0pt}{}{\mathbf 1}{\mathbf 2}\oplus\mathbf 2
 \longrightarrow
 \genfrac{}{}{0pt}{}{\mathbf 1}{\mathbf 2}
 \oplus
 \genfrac{}{}{0pt}{}{\mathbf 1}{\mathbf 2},
 \qquad
 c(x,y)=(x,\alpha y).
\]
Its cokernel is $\mathbf 1$.  Hence Proposition~\ref{prop:module-coBongartz}
gives
\[
 G_c=\genfrac{}{}{0pt}{}{\mathbf 1}{\mathbf 2}\oplus\mathbf 1
 \cong P(\Fac_\Gamma K).
\]
In particular, the cokernel term in
Proposition~\ref{prop:module-coBongartz} can contribute a genuinely new
indecomposable summand.
\end{example}

\begin{theorem}[Finite bilateral approximation gluing]
\label{thm:bilateral}
Assume that both correction ladders terminate, say at $E_X^{r_X}$ and
$E_Y^{r_Y}$.  Let
\[
 \alpha_X:F_AX\longrightarrow E_X^{r_X},
 \qquad
 \alpha_Y:F_BY\longrightarrow E_Y^{r_Y}
\]
be the corresponding composites, and let
\[
 \overline\alpha_X:F_AX\longrightarrow K_X,
 \qquad
 \overline\alpha_Y:F_BY\longrightarrow K_Y
\]
be their minimal parts.
Let $K$ be obtained from $K_X\oplus K_Y$ by deleting repeated isomorphic
indecomposable summands.  Then:
\begin{enumerate}[label=\textup{(\arabic*)},leftmargin=2.2em]
 \item One has
 \[
  \Ext^1_\Lambda(K_X\oplus K_Y,\mathcal T_{X,Y})=0,
  \qquad
  \Fac_\Lambda(K_X\oplus K_Y)=\mathcal T_{X,Y}.
 \]
 \item The module $K$ is $\tau$-rigid, and $\mathcal T_{X,Y}$ is
 functorially finite.
 \item For any left $\addc K$-approximation $c:\Lambda\to K'$, one has
 \[
  G_{\mathrm{bi}}(X,Y):=K\oplus\Coker c
  \cong P(\mathcal T_{X,Y}).
 \]
 Thus the resulting basic module is independent of the choice of $c$ up
 to isomorphism.  In particular, $G_{\mathrm{bi}}(X,Y)$ is support
 $\tau$-tilting and
 $\Fac_\Lambda G_{\mathrm{bi}}(X,Y)=\mathcal T_{X,Y}$.
\end{enumerate}
\end{theorem}

\begin{proof}
\textup{(1)} By Lemma~\ref{lem:elementary-correction}, every arrow in the
two correction ladders is respectively a left
$\mathcal C_{\mathcal X}$- or left
$\mathcal C_{\mathcal Y}$-approximation.  Since
$\mathcal T_{X,Y}\subseteq\mathcal C_{\mathcal X}\cap
\mathcal C_{\mathcal Y}$, every morphism from a term of either ladder to
an object of $\mathcal T_{X,Y}$ factors through the next arrow.  Iterating
this factorization shows that the terminating composites $\alpha_X$ and
$\alpha_Y$ are left $\mathcal T_{X,Y}$-approximations.  Their terminal
targets lie in $\mathcal T_{X,Y}$ by definition of termination.  Passing
to the minimal parts preserves the approximation property, and
$K_X,K_Y\in\mathcal T_{X,Y}$ because a torsion class is closed under
direct summands.  Hence $\overline\alpha_X$ and
$\overline\alpha_Y$ are minimal left $\mathcal T_{X,Y}$-approximations.
Lemma~\ref{lem:corner-Ext-projective} and the relative Wakamatsu lemma,
Lemma~\ref{lem:Wakamatsu}, give
\[
 \Ext^1_\Lambda(K_X,\mathcal T_{X,Y})=0
 =\Ext^1_\Lambda(K_Y,\mathcal T_{X,Y}),
\]
and therefore
$\Ext^1_\Lambda(K_X\oplus K_Y,\mathcal T_{X,Y})=0$.

It remains to show that $K_X\oplus K_Y$ generates $\mathcal T_{X,Y}$.
Let $W=(U,V,u,v)\in\mathcal T_{X,Y}$.  Since
$U\in\Fac_AX$ and $V\in\Fac_BY$, choose epimorphisms
\[
 \alpha:X^{\oplus r}\twoheadrightarrow U,
 \qquad
 \beta:Y^{\oplus s}\twoheadrightarrow V.
\]
Under the adjunctions $F_A\dashv(-)e$ and $F_B\dashv(-)f$, they correspond
to morphisms
\[
 \widetilde\alpha:F_AX^{\oplus r}\longrightarrow W,
 \qquad
 \widetilde\beta:F_BY^{\oplus s}\longrightarrow W,
\]
whose components are
\begin{align*}
 \widetilde\alpha_e&=\alpha,
 &\widetilde\alpha_f&=u(\alpha\otimes1_N),\\
 \widetilde\beta_e&=v(\beta\otimes1_M),
 &\widetilde\beta_f&=\beta.
\end{align*}
Thus the $A$-component of the row morphism
$\bigl[\,\widetilde\alpha\quad\widetilde\beta\,\bigr]$ is
\[
 \bigl[\,\alpha\quad v(\beta\otimes1_M)\,\bigr]:
 X^{\oplus r}\oplus(Y^{\oplus s}\otimes_BM)\longrightarrow U,
\]
which is epic because its restriction to $X^{\oplus r}$ is $\alpha$.
Similarly, its $B$-component is
\[
 \bigl[\,u(\alpha\otimes1_N)\quad\beta\,\bigr]:
 (X^{\oplus r}\otimes_AN)\oplus Y^{\oplus s}\longrightarrow V,
\]
which is epic because its restriction to $Y^{\oplus s}$ is $\beta$.
Since exactness in $\modu\Lambda$ is detected componentwise,
$\bigl[\,\widetilde\alpha\quad\widetilde\beta\,\bigr]$ is an epimorphism.

The maps $\overline\alpha_X^{\oplus r}$ and
$\overline\alpha_Y^{\oplus s}$ are again left
$\mathcal T_{X,Y}$-approximations.  Hence $\widetilde\alpha$ and
$\widetilde\beta$ factor through them, so there is a morphism
\[
 g:K_X^{\oplus r}\oplus K_Y^{\oplus s}\longrightarrow W
\]
such that
\[
 g(\overline\alpha_X^{\oplus r}\oplus
   \overline\alpha_Y^{\oplus s})
 =\bigl[\,\widetilde\alpha\quad\widetilde\beta\,\bigr].
\]
The right-hand side is epic, hence $g$ is epic.  Therefore every object of
$\mathcal T_{X,Y}$ belongs to
$\Fac_\Lambda(K_X\oplus K_Y)$.  The reverse inclusion follows from
$K_X,K_Y\in\mathcal T_{X,Y}$ and the fact that $\mathcal T_{X,Y}$ is a
torsion class.  This proves
\[
 \Fac_\Lambda(K_X\oplus K_Y)=\mathcal T_{X,Y},
\]
and hence \textup{(1)}.

\textup{(2)} By construction,
$\addc K=\addc(K_X\oplus K_Y)$.  Part~\textup{(1)} therefore gives
\[
 \Fac_\Lambda K=\mathcal T_{X,Y},
 \qquad
 \Ext^1_\Lambda(K,\Fac_\Lambda K)=0.
\]
Theorem~\ref{thm:AIR-torsion}\textup{(1)} shows that $K$ is
$\tau$-rigid and that $\mathcal T_{X,Y}$ is functorially finite.  This
proves \textup{(2)}.

\textup{(3)} Let $c:\Lambda\to K'$ be any left
$\addc K$-approximation.  Proposition~\ref{prop:module-coBongartz}, applied
to the $\tau$-rigid module $K$, gives
\[
 K\oplus\Coker c
 \cong P(\Fac_\Lambda K)
 =P(\mathcal T_{X,Y}).
\]
Thus $G_{\mathrm{bi}}(X,Y)$ is support $\tau$-tilting, its isomorphism
class is independent of the chosen left $\addc K$-approximation, and
$\Fac_\Lambda G_{\mathrm{bi}}(X,Y)=\mathcal T_{X,Y}$.  This proves
\textup{(3)}.
\end{proof}

\begin{remark}[Choice of the completion map]
\label{rem:completion-map}
Both Proposition~\ref{prop:module-coBongartz} and
Theorem~\ref{thm:bilateral} allow an arbitrary left
$\addc K$-approximation.  For computations one may take a minimal one,
say $c_{\min}:\Gamma\to K_{\min}$.  There is also a canonical choice.
Give $K\otimes_kDK$ the right $\Gamma$-action on the first tensor factor.
For a basis $(z_i)_{i=1}^d$ of $K$ with dual basis
$(z_i^*)_{i=1}^d$, set
\begin{equation}\label{eq:canonical-completion-map}
 \eta_K:\Gamma\longrightarrow K\otimes_kDK,
 \qquad
 \eta_K(\gamma)=\sum_{i=1}^d z_i\gamma\otimes z_i^*.
\end{equation}
Under the canonical isomorphism
\[
 \Phi:K\otimes_kDK\longrightarrow\operatorname{End}_k(K),
 \qquad
 \Phi(x\otimes\xi)(y)=\xi(y)x,
\]
the element $\varepsilon_K:=\sum_i z_i\otimes z_i^*$ corresponds to
$1_K$, since
$\Phi(\varepsilon_K)(y)=\sum_i z_i^*(y)z_i=y$.  Hence
$\varepsilon_K=\Phi^{-1}(1_K)$, so it is independent of the chosen basis;
consequently $\eta_K(\gamma)=\varepsilon_K\gamma$ is basis-independent as
well.  Moreover,
$K\otimes_kDK\cong K^{\oplus d}$ as a right $\Gamma$-module.

If $h:\Gamma\to K$ and $y=h(1)$, then
$\theta_y(x\otimes\xi)=\xi(y)x$ is $\Gamma$-linear and satisfies
$\theta_y\eta_K=h$.  Hence $\eta_K$ is a left
$\addc K$-approximation.  By the standard decomposition of approximations
recalled after Definition~\ref{def:approximations},
$\eta_K\cong(c_{\min},0)$ after an isomorphism of targets.  Thus
\[
 \Coker\eta_K\cong\Coker c_{\min}\oplus K_0
\]
for some $K_0\in\addc K$, and the minimal and canonical choices give the
same basic completion under our standing convention.
\end{remark}

\begin{example}[Finite genuinely bilateral correction]
\label{ex:bilateral-finite}
Let $\Lambda=kQ$, where
\[
\begin{tikzcd}[column sep=4em,row sep=2em]
 1 \arrow[r,"\alpha"] \arrow[d,"n"'] & 2 \\
 4 & 3 \arrow[l,"\beta"] \arrow[r,"m"] & 5
\end{tikzcd}
\]
Put $e=\varepsilon_1+\varepsilon_2+\varepsilon_5$ and
$f=\varepsilon_3+\varepsilon_4$.  Then $A=e\Lambda e$ is the path algebra
of $1\xrightarrow{\alpha}2$ together with the isolated vertex $5$, whereas
$B=f\Lambda f=k(3\xrightarrow{\beta}4)$.  Moreover,
\[
 N=e\Lambda f=kn,
 \qquad M=f\Lambda e=km,
 \qquad N\otimes_BM=0=M\otimes_AN.
\]
Write $\mathbf i$ for the simple at vertex $i$, and set
\[
 U=\genfrac{}{}{0pt}{}{\mathbf1}{\mathbf2},
 \qquad
 V=\genfrac{}{}{0pt}{}{\mathbf3}{\mathbf4},
 \qquad
 X=U\oplus\mathbf1,
 \qquad
 Y=V\oplus\mathbf3.
\]
Since vertex $5$ is isolated, $\mathbf5$ is projective and
$(X,\mathbf5)$ is a support $\tau$-tilting pair over $A$; moreover, $Y$ is
$\tau$-tilting over $B$.  Thus
\[
 \mathcal X=\Fac_AX=\addc(U\oplus\mathbf1),
 \qquad
 \mathcal Y=\Fac_BY=\addc(V\oplus\mathbf3).
\]
On the other hand,
\[
 X\otimes_AN\cong\mathbf4^{\oplus2}\notin\mathcal Y,
 \qquad
 Y\otimes_BM\cong\mathbf5^{\oplus2}\notin\mathcal X.
\]
Hence neither one-sided compatibility condition holds.

The minimal corner approximations needed by the two ladders are
\[
 \mathbf4^{\oplus2}\lhook\joinrel\longrightarrow V^{\oplus2},
 \qquad
 X\oplus\mathbf5^{\oplus2}
   \xrightarrow{[\,1_X\ \ 0\,]}X,
 \qquad
 \mathbf5^{\oplus2}\longrightarrow0.
\]
The resulting correction terms are displayed below.  The middle columns
list only the two corner components; their structure maps are the maps
produced by the pushouts.
\[
\begin{array}{c|c|c|c}
 \text{term}&e\text{-component}&f\text{-component}&\text{next step}
 \rule{0pt}{2.6ex}\\ \hline
 E_X^0&X&\mathbf4^{\oplus2}&\mathcal Y\text{-correction}
 \rule{0pt}{3.1ex}\\
 E_X^1&X\oplus\mathbf5^{\oplus2}&V^{\oplus2}
      &\mathcal X\text{-correction}\rule{0pt}{3.1ex}\\
 E_X^2&X&V^{\oplus2}&\text{terminal}\rule{0pt}{3.1ex}\\ \hline
 E_Y^0&\mathbf5^{\oplus2}&Y&\mathcal X\text{-correction}
 \rule{0pt}{3.1ex}\\
 E_Y^1&0&Y&\text{terminal}\rule{0pt}{3.1ex}
\end{array}
\]
Indeed, in the first ladder the common source of the first pushout is
$\mathbf4^{\oplus2}\otimes_BM=0$, so the new $A$-component is
$X\oplus(V^{\oplus2}\otimes_BM)\cong X\oplus\mathbf5^{\oplus2}$.
The next approximation projects away this new summand, and its tensor with
$N$ is the identity on $X\otimes_AN$; hence the next pushout leaves the
$B$-component $V^{\oplus2}$ unchanged.  In $E_X^1$ the $n$-map is the
radical inclusion $\mathbf4^{\oplus2}\hookrightarrow V^{\oplus2}$ and the
$m$-map identifies $V^{\oplus2}\otimes_BM$ with the new
$\mathbf5^{\oplus2}$-summand; after the second correction the $m$-map is
zero.  For the second ladder,
$\mathbf5^{\oplus2}\otimes_AN=0$ and its minimal left
$\mathcal X$-approximation is zero, giving $E_Y^1=(0,Y)$.

Thus the $X$-ladder terminates after two corrections and the $Y$-ladder
after one.  This is a genuinely bilateral instance of
Theorem~\ref{thm:bilateral}: both one-sided hypotheses fail, but both
alternating ladders terminate.

The resulting support $\tau$-tilting module is also explicit.  Let $L_U$
and $L_{\mathbf1}$ be the $\Lambda$-modules whose $e$-components are $U$
and $\mathbf1$, respectively, whose $f$-components are both $V$, and whose
$n$-maps are the radical inclusions $\mathbf4\hookrightarrow V$; their
$m$-maps are zero.  Then
\[
 E_X^2\cong L_U\oplus L_{\mathbf1},
 \qquad
 E_Y^1\cong V\oplus\mathbf3,
\]
where the two modules on the right in the second isomorphism are regarded
as $\Lambda$-modules supported on the $f$-corner.  Under these
decompositions, a direct endomorphism check shows that the terminal
composites are minimal.  Hence the module $K$ in
Theorem~\ref{thm:bilateral} is
\[
 K=L_U\oplus L_{\mathbf1}\oplus V\oplus\mathbf3.
\]
The four summands are indecomposable and have distinct dimension vectors.
By Theorem~\ref{thm:bilateral}, $K$ is $\tau$-rigid and
\[
 \Fac_\Lambda K=\mathcal T_{X,Y}.
\]
Moreover, $K$ is zero at vertex $5$, so
$\Hom_\Lambda(\mathbf5,K)=0$.  Since $\mathbf5=e_5\Lambda$ is projective,
$\Lambda$ has five simple modules, and $K$ has four indecomposable summands,
$(K,\mathbf5)$ is a support $\tau$-tilting pair.  Consequently, the
completion adds no new basic summand,
and
\[
 G_{\mathrm{bi}}(X,Y)\cong K.
\]
\end{example}

\subsection{The limitation of finite termination}

We first recall the order-theoretic facts used in this subsection.  Let $P$
be a poset and let $S\subseteq P$.  A \emph{meet} of $S$, if it exists, is
its greatest lower bound; dually, a \emph{join} of $S$ is its least upper
bound.  The poset $P$ is a \emph{lattice} if every finite subset has both a
meet and a join \cite[Definition~2.2]{IRTT}.

For a finite-dimensional algebra $C$, let
$\mathrm{f\text{-}tors}\,C$ be the poset of functorially finite torsion
classes in $\modu C$, ordered by inclusion.  It has least element $0$ and
greatest element $\modu C$.  Hence it is a lattice precisely when every pair
of its elements has both a meet and a join.  If
$\mathcal U,\mathcal V\in\mathrm{f\text{-}tors}\,C$ and
$\mathcal U\cap\mathcal V$ is functorially finite, then this intersection is
their meet.  The join, when it exists, is the least functorially finite
torsion class containing both $\mathcal U$ and $\mathcal V$; it need not be
their set-theoretic union.

The lattice problem for functorially finite torsion classes was studied by
Iyama, Reiten, Thomas and Todorov.  For the path algebra $kQ$ of a finite
connected acyclic quiver over an algebraically closed field, they proved that
$\mathrm{f\text{-}tors}\,kQ$ is a lattice precisely when $Q$ is Dynkin or has
at most two vertices \cite[Theorem~1.3]{IRTT}.  Ringel subsequently extended
this result to connected hereditary Artin algebras: for such an algebra $H$,
$\mathrm{f\text{-}tors}\,H$ is a lattice precisely when $H$ is
representation-finite or has exactly two simple modules
\cite[Theorem]{Ringel}.  We use Ringel's version below.

We now show that one cannot remove all conditions while preserving the
original componentwise class.

\begin{proposition}[Obstruction to unconditional termination]
\label{prop:obstruction}
For every field $k$, there exist a finite-dimensional Morita context
$k$-algebra $\Lambda$ and support $\tau$-tilting corner modules $X,Y$ for
which
$\mathcal T_{X,Y}$ is not functorially finite.  Consequently, no
condition-free finite bilateral construction can always produce a support
$\tau$-tilting module $G$ satisfying
$\Fac_\Lambda G=\mathcal T_{X,Y}$.
\end{proposition}

\begin{proof}
Let $Q$ be a connected acyclic Euclidean quiver with at least three
vertices, and put $H=kQ$.  Then $H$ is a connected representation-infinite
hereditary Artin algebra with at least three simple modules.  By Ringel's
main theorem, $\mathrm{f\text{-}tors}\,H$ is not a lattice
\cite[Theorem]{Ringel}.  Hence some pair in
$\mathrm{f\text{-}tors}\,H$ has no meet or no join.

If there exist $\mathcal U,\mathcal V\in\mathrm{f\text{-}tors}\,H$ having
no meet, set $C=H$, $\mathcal X=\mathcal U$ and
$\mathcal Y=\mathcal V$.  Otherwise, there exist
$\mathcal U,\mathcal V\in\mathrm{f\text{-}tors}\,H$ having no join.  The
$k$-duality, applied to the torsion pair
$(\mathcal T,\mathcal T^{\perp})$, gives an order anti-isomorphism
\[
 \Psi_H:\mathrm{f\text{-}tors}\,H
 \longrightarrow \mathrm{f\text{-}tors}\,H^{\mathrm{op}},
 \qquad
 \mathcal T\longmapsto D(\mathcal T^{\perp}),
\]
where
$\mathcal T^{\perp}=\{L\mid\Hom_H(\mathcal T,L)=0\}$.
Finite-dimensional duality and the two halves of a torsion pair give this
anti-isomorphism; functorial finiteness passes between the two halves by
\cite[Proposition~2.1(a),(b)]{Jasso}.  Its inverse is
$\mathcal S\mapsto{}^{\perp}D\mathcal S$, so no algebraic-closure hypothesis
is involved.  Set $C=H^{\mathrm{op}}$,
$\mathcal X=\Psi_H(\mathcal U)$ and
$\mathcal Y=\Psi_H(\mathcal V)$.  These two classes have no meet: otherwise,
the inverse image of their meet under $\Psi_H$ would be the join of
$\mathcal U$ and $\mathcal V$.  Thus, in either case,
$\mathcal X\cap\mathcal Y$ is not functorially finite: otherwise it would be
the meet of $\mathcal X$ and $\mathcal Y$ by the preceding paragraph.
By Theorem~\ref{thm:AIR-torsion}\textup{(2)}, choose support
$\tau$-tilting $C$-modules $X$ and $Y$ such that
\[
 \mathcal X=\Fac_CX,
 \qquad
 \mathcal Y=\Fac_CY.
\]

Take the strict Morita context
\[
 \Lambda=M_2(C),\qquad A=B=C,\qquad M=N=C,
\]
with both connecting maps given by multiplication.  Under the standard
Morita equivalence $\modu M_2(C)\simeq\modu C$, the two corner restrictions
of a $\Lambda$-module are naturally isomorphic to the same $C$-module.
Indeed, right multiplication by the matrix units $e_{12}$ and $e_{21}$
gives mutually inverse maps between the $e_{11}$- and $e_{22}$-corners.
Therefore $\mathcal T_{X,Y}$ corresponds precisely to
$\mathcal X\cap\mathcal Y$ and is not functorially finite.  If a support
$\tau$-tilting module $G$ had
$\Fac_\Lambda G=\mathcal T_{X,Y}$, the latter would be functorially finite,
a contradiction.  Since $\mathcal T_{X,Y}$ is not functorially finite,
Theorem~\ref{thm:bilateral} implies that the two correction ladders cannot
both terminate for this pair.
\end{proof}

The following example shows that the finite termination hypothesis in
Theorem~\ref{thm:bilateral} cannot be omitted even under the strong condition
\[
 N\otimes_B M=0=M\otimes_A N.
\]

We recall the Euclidean-quiver facts used in the next example.  Let $Q'$ be
an acyclic Euclidean quiver, let $\delta_{Q'}$ be the minimal positive vector
spanning the radical of its Tits form, and write
\[
 \langle x,y\rangle_{Q'}
 =\sum_{i\in Q'_0}x_iy_i-
   \sum_{a:i\to j\text{ in }Q'}x_iy_j
\]
for the Euler form.  The \emph{defect} of a representation $W$ is
\[
 \partial_{Q'}(W)=
 \langle\delta_{Q'},\underline{\dim}W\rangle_{Q'}.
\]
The Dlab--Ringel defect criterion says that an indecomposable representation
is preprojective, regular, or preinjective according as its defect is negative,
zero, or positive.  Moreover, the regular components form a separating family
of stable tubes: distinct tubes are Hom- and Ext-orthogonal, and there are no
nonzero maps from a preinjective module to a regular module.  These facts hold
for tame hereditary Artin algebras, hence for $kQ'$ over an arbitrary field; see
\cite[Sections~1, 3 and~6]{DlabRingel} and
\cite[Section~3, Case~1]{Ringel}.

\begin{example}[Nontermination with vanishing cross tensor products]
\label{ex:zero-products}
Over the arbitrary ground field $k$, let
\[
 A=k(1\xrightarrow{\alpha}2),
 \qquad
 B=k(3\xrightarrow{\beta}4),
\]
and write $\varepsilon_i$ for the vertex idempotent at $i$ in the
corresponding corner algebra.  Let ${}_AN_B=kn$ and ${}_BM_A=km$ be the
one-dimensional bimodules determined by
\[
 \varepsilon_1n=n=n\varepsilon_4,
 \qquad
 \varepsilon_3m=m=m\varepsilon_2,
\]
with every other idempotent action and every radical action equal to zero.
The support idempotents on the two sides do not match, and hence
\[
 n\otimes m=n\varepsilon_4\otimes m
   =n\otimes\varepsilon_4m=0,
 \qquad
 m\otimes n=m\varepsilon_2\otimes n
   =m\otimes\varepsilon_2n=0.
\]
Thus $N\otimes_BM=0=M\otimes_AN$.  Take $\phi=0=\psi$.  The resulting
Morita context algebra is the path algebra $\Lambda=kQ$, where
\[
\begin{tikzcd}[column sep=4em,row sep=2em]
 1 \arrow[r,"\alpha"] \arrow[d,"n"'] & 2 \\
 4 & 3 \arrow[l,"\beta"] \arrow[u,"m"']
\end{tikzcd}
\]
In particular, $Q$ is a bipartite Euclidean quiver of type
$\widetilde A_{2,2}$ (with underlying diagram $\widetilde A_3$); indeed, every
arrow goes from $\{1,3\}$ to $\{2,4\}$, so $J(\Lambda)^2=0$.

Write $\mathbf i$ for the simple module at vertex $i$, viewed over the
appropriate corner algebra, and set
\[
 U=\genfrac{}{}{0pt}{}{\mathbf 1}{\mathbf 2},
 \qquad
 V=\genfrac{}{}{0pt}{}{\mathbf 3}{\mathbf 4}.
\]
Thus $U=\varepsilon_1A$ and $V=\varepsilon_3B$ are the indecomposable
projective modules in the two corners.  Put
\[
 X=\genfrac{}{}{0pt}{}{\mathbf 1}{\mathbf 2}\oplus\mathbf 1,
 \qquad
 Y=\genfrac{}{}{0pt}{}{\mathbf 3}{\mathbf 4}\oplus\mathbf 3.
\]
Since $A$ and $B$ are both of Dynkin type $A_2$, $X$ is a support
$\tau$-tilting $A$-module and $Y$ is a support $\tau$-tilting $B$-module.
Moreover,
\[
 \mathcal X=\Fac_AX
 =\addc\left(
   \genfrac{}{}{0pt}{}{\mathbf 1}{\mathbf 2}\oplus\mathbf 1
 \right),
 \qquad
 \mathcal Y=\Fac_BY
 =\addc\left(
   \genfrac{}{}{0pt}{}{\mathbf 3}{\mathbf 4}\oplus\mathbf 3
 \right).
\]
An $A$-module belongs to $\mathcal X$ precisely when its linear map along
$\alpha$ is surjective; this is just the assertion that it has no direct
summand isomorphic to $\mathbf 2$.  The analogous description holds for
$\mathcal Y$, with excluded simple $\mathbf 4$.
Consequently, $\mathcal T_{X,Y}$ consists exactly of the representations of
$Q$ for which the maps along both $\alpha$ and $\beta$ are surjective.

\textup{(1)} We first prove, still over an arbitrary field, that this torsion
class is not functorially finite.  Regard $U$ and $V$ as $\Lambda$-modules by
extending them by zero to the other two vertices, and put
$P_i=\varepsilon_i\Lambda$.  Here $\delta_Q=(1,1,1,1)$, and hence
\[
 \partial_Q(W)=
 \dim_kW_1+\dim_kW_3-\dim_kW_2-\dim_kW_4.
\]
In particular, $\partial_Q(U)=0=\partial_Q(V)$.  Their minimal projective
presentations are
\[
 0\longrightarrow P_4\xrightarrow{\,\varepsilon_4\mapsto n\,}P_1
   \longrightarrow U\longrightarrow0,
 \qquad
 0\longrightarrow P_2\xrightarrow{\,\varepsilon_2\mapsto m\,}P_3
   \longrightarrow V\longrightarrow0.
\]
After applying $\Hom_\Lambda(-,\Lambda)$, the two relevant maps have images
$kn\subseteq\Lambda\varepsilon_4$ and
$km\subseteq\Lambda\varepsilon_2$.  Therefore
\[
 \tau U=D(\Lambda\varepsilon_4/kn)\cong V,
 \qquad
 \tau V=D(\Lambda\varepsilon_2/km)\cong U.
\]
The same presentations give
\[
 \Ext^1_\Lambda(U,V)\cong k
 \cong\Ext^1_\Lambda(V,U).
\]
The corresponding non-split extensions have one-dimensional spaces at all
four vertices.  In the first, $\alpha,\beta,n$ are identities and $m=0$;
in the second, $\alpha,\beta,m$ are identities and $n=0$.  Each middle term
is indecomposable, since an endomorphism is multiplication by the same scalar
at every vertex.  These are therefore the almost split sequences ending in
$U$ and $V$.

Thus $U$ and $V$ are regular.  The preceding calculation places them at the
mouth of the same stable tube, and
$\tau U\cong V$, $\tau V\cong U$ show that its rank is $2$.  Let
$\mathcal R$ denote the full additive subcategory of this tube.  Its standard
serial structure gives
\[
 \mathcal R=\operatorname{Filt}(U,V);
\]
viewed as a length category, its simple objects are $U$ and $V$.  Write
$\ell_{\mathcal R}$ for Loewy length in this category.  It is nonincreasing
under quotients and
$\ell_{\mathcal R}(L^{\oplus t})=\ell_{\mathcal R}(L)$.  Moreover, for every
$r\geq1$ there is an indecomposable $M_r\in\mathcal R$ with
$\ell_{\mathcal R}(M_r)=r$; see \cite[Section~3, Case~1]{Ringel}.
Thus each object has finite length, while these lengths are unbounded.  Since
$U,V\in\mathcal T_{X,Y}$ and $\mathcal T_{X,Y}$ is extension closed, we have
$\mathcal R\subseteq\mathcal T_{X,Y}$.

Suppose, to the contrary, that $\mathcal T_{X,Y}$ is functorially finite.
By Theorem~\ref{thm:AIR-torsion}\textup{(2)}, there is a support
$\tau$-tilting module $G\in\mathcal T_{X,Y}$ such that
$\mathcal T_{X,Y}=\Fac_\Lambda G$.
If $W\in\mathcal T_{X,Y}$, the surjectivity of its maps along $\alpha$ and
$\beta$ gives $\partial_Q(W)\geq0$.  Hence every indecomposable direct
summand of $G$ is regular or preinjective; no preprojective summand can
occur.  Let $G_{\mathcal R}$ be the direct sum of the indecomposable
summands of $G$ lying in $\mathcal R$, and set
$L=\ell_{\mathcal R}(G_{\mathcal R})$, with $L=0$ if
$G_{\mathcal R}=0$.  Distinct regular tubes are
Hom-orthogonal, and there are no nonzero morphisms from a preinjective module
to a regular module.  Thus every morphism from a finite direct sum of copies
of $G$ to an object of $\mathcal R$ factors through a finite direct sum of
copies of $G_{\mathcal R}$.

Choose an indecomposable object $M'\in\mathcal R$ with
$\ell_{\mathcal R}(M')=L+1$.  Since
$\mathcal R\subseteq\mathcal T_{X,Y}=\Fac_\Lambda G$, there is an
epimorphism $G^{\oplus t}\twoheadrightarrow M'$ for some $t$.  By the
preceding orthogonality, this epimorphism factors through the canonical
projection $G^{\oplus t}\twoheadrightarrow G_{\mathcal R}^{\oplus t}$, so the
induced map $G_{\mathcal R}^{\oplus t}\twoheadrightarrow M'$ is also epic.
Hence
\[
 \ell_{\mathcal R}(M')
 \leq \ell_{\mathcal R}(G_{\mathcal R}^{\oplus t})
 =L,
\]
a contradiction.  Therefore $\mathcal T_{X,Y}$ is not functorially finite.

\textup{(2)} The failure of termination can also be seen directly.  The
relevant tensor products and minimal left approximations are
\[
 X\otimes_AN\cong\mathbf 4^{\oplus2},
 \qquad
 \mathbf 4^{\oplus2}\lhook\joinrel\longrightarrow
 \left(
   \genfrac{}{}{0pt}{}{\mathbf 3}{\mathbf 4}
 \right)^{\oplus2},
\]
and, symmetrically,
\[
 Y\otimes_BM\cong\mathbf 2^{\oplus2},
 \qquad
 \mathbf 2^{\oplus2}\lhook\joinrel\longrightarrow
 \left(
   \genfrac{}{}{0pt}{}{\mathbf 1}{\mathbf 2}
 \right)^{\oplus2}.
\]
The initial term $E_X^0=F_AX$ has corner components
$X$ and $\mathbf 4^{\oplus2}$.  Since
$\mathbf 4\notin\mathcal Y$, the first $\mathcal Y$-correction embeds
$\mathbf 4^{\oplus2}$ into
$\left(\genfrac{}{}{0pt}{}{\mathbf 3}{\mathbf 4}\right)^{\oplus2}$.
Moreover,
\[
 (X\otimes_AN)\otimes_BM
 \cong X\otimes_A(N\otimes_BM)=0.
\]
Consequently, the defining pushout is a direct sum: it replaces the unwanted
$\mathbf 4^{\oplus2}$ and adds
$\left(\genfrac{}{}{0pt}{}{\mathbf 3}{\mathbf 4}\right)^{\oplus2}
 \otimes_BM\cong\mathbf 2^{\oplus2}$
to the $A$-component.  The new summand $\mathbf 2^{\oplus2}$ does not belong
to $\mathcal X$, so the next $\mathcal X$-correction performs the symmetric
operation.  The first four terms are therefore as follows; the last column
records the correction that is still required.
\[
\begin{array}{c|c|c|c}
 i &(E_X^i)e &(E_X^i)f &\text{next correction}\rule{0pt}{2.5ex}\\ \hline
 0
 &X
 &\mathbf 4^{\oplus2}
 &\mathcal Y\rule{0pt}{3.2ex}\\
 1
 &X\oplus\mathbf 2^{\oplus2}
 &\left(\genfrac{}{}{0pt}{}{\mathbf 3}{\mathbf 4}\right)^{\oplus2}
 &\mathcal X\rule{0pt}{4.2ex}\\
 2
 &X\oplus
   \left(\genfrac{}{}{0pt}{}{\mathbf 1}{\mathbf 2}\right)^{\oplus2}
 &\left(\genfrac{}{}{0pt}{}{\mathbf 3}{\mathbf 4}\right)^{\oplus2}
   \oplus\mathbf 4^{\oplus2}
 &\mathcal Y\rule{0pt}{4.2ex}\\
 3
 &X\oplus
   \left(\genfrac{}{}{0pt}{}{\mathbf 1}{\mathbf 2}\right)^{\oplus2}
   \oplus\mathbf 2^{\oplus2}
 &\left(\genfrac{}{}{0pt}{}{\mathbf 3}{\mathbf 4}\right)^{\oplus4}
 &\mathcal X\rule{0pt}{4.2ex}
\end{array}
\]
The table exhibits the alternating pattern.  Induction gives, for every
$r\geq0$,
\begin{align*}
 (E_X^{2r})e
   &\cong X\oplus
   \left(
     \genfrac{}{}{0pt}{}{\mathbf 1}{\mathbf 2}
   \right)^{\oplus2r},
 & (E_X^{2r})f
   &\cong
   \left(
     \genfrac{}{}{0pt}{}{\mathbf 3}{\mathbf 4}
   \right)^{\oplus2r}
   \oplus\mathbf 4^{\oplus2},\\
 (E_X^{2r+1})e
   &\cong X\oplus
   \left(
     \genfrac{}{}{0pt}{}{\mathbf 1}{\mathbf 2}
   \right)^{\oplus2r}
   \oplus\mathbf 2^{\oplus2},
 & (E_X^{2r+1})f
   &\cong
   \left(
     \genfrac{}{}{0pt}{}{\mathbf 3}{\mathbf 4}
   \right)^{\oplus(2r+2)}.
\end{align*}
Every even term has a $\mathbf 4$-summand in its $B$-component and hence
does not belong to $\mathcal C_{\mathcal Y}$; every odd term has a
$\mathbf 2$-summand in its $A$-component and hence does not belong to
$\mathcal C_{\mathcal X}$.  Thus the $X$-correction ladder never
terminates.  In other words, the vanishing of $N\otimes_BM$ and
$M\otimes_AN$ makes these pushouts split, but the nonzero one-step tensors
continue to create a new excluded simple at the opposite corner.
Interchanging the two corners gives the same conclusion for the
$Y$-correction ladder.

\end{example}

\section{One-sided approximation gluing}

\subsection{The explicit one-step correction}

Assume the one-sided compatibility condition
\begin{equation*}\tag{C$_{\mathcal X}$}
 X\otimes_AN\in\mathcal Y.
\end{equation*}
Since tensor products are right exact, this implies
\begin{equation}\label{eq:tensor-all-X}
 U\otimes_AN\in\mathcal Y
 \qquad\text{for every }U\in\mathcal X.
\end{equation}
Let
\[
 p:Y\otimes_BM\longrightarrow X_Y
\]
be a minimal left $\mathcal X$-approximation.  The possible correction in
the $Y$-ladder is the specialization of $R_{\mathcal X}$ to $F_BY$.
We spell out this specialization because it
gives the explicit module used below.  Its $B$-component is the pushout
\begin{equation}\label{eq:pushout}
\begin{tikzcd}[column sep=4.5em,row sep=3.0em]
 Y\otimes_BM\otimes_AN
   \arrow[r,"1_Y\otimes\psi"]
   \arrow[d,"p\otimes1_N"']
 &Y\arrow[d,"\lambda"]\\
 X_Y\otimes_AN\arrow[r,"\mu"']
 &V_p.
\end{tikzcd}
\end{equation}
Thus
\[
 V_p=
 \frac{Y\oplus(X_Y\otimes_AN)}
 {\operatorname{Im}(1_Y\otimes\psi,-p\otimes1_N)}.
\]
Applying the construction of Section~3.1 to
\[
 F_BY=(Y\otimes_BM,Y,1_Y\otimes\psi,1_{Y\otimes M})
\]
gives a unique $A$-module homomorphism
\[
 \overline p:V_p\otimes_BM\longrightarrow X_Y
\]
satisfying
\begin{align*}
 \overline p(\lambda(y)\otimes m)&=p(y\otimes m),\\
 \overline p(\mu(x\otimes n)\otimes m)&=x\phi(n\otimes m).
\end{align*}
Set
\[
 K_p=(X_Y,V_p,\mu,\overline p),
 \qquad
 a_p=(p,\lambda):F_BY\longrightarrow K_p.
\]
Proposition~\ref{prop:well-defined-corrections} shows that $K_p$ is a
$\Lambda$-module and that $a_p$ is its canonical correction morphism.
Since $X_Y\in\mathcal X$, formula \eqref{eq:tensor-all-X} gives
$X_Y\otimes_AN\in\mathcal Y$.  The module $V_p$ is a quotient of
$Y\oplus(X_Y\otimes_AN)$, so $V_p\in\mathcal Y$.  Hence
$K_p\in\mathcal T_{X,Y}$.
By Lemma~\ref{lem:elementary-correction}, $a_p$ is a minimal left
$\mathcal C_{\mathcal X}$-approximation.  Since
$\mathcal T_{X,Y}\subseteq\mathcal C_{\mathcal X}$ and its target belongs
to $\mathcal T_{X,Y}$, it is also a minimal left
$\mathcal T_{X,Y}$-approximation.  Thus the $Y$-ladder terminates after at
most one correction, with terminal module isomorphic to $K_p$; the
$X$-ladder has already terminated at $F_AX$.  If $F_BY$ already belongs to
$\mathcal T_{X,Y}$, we may take $p$ to be the identity and then
$K_p\cong F_BY$.

\subsection{Relative presentations and the one-sided theorem}

The only point not supplied by the bilateral theorem is that no additional
Ext-projective summand is needed.  The following relative presentation
property establishes this.

\begin{proposition}[One-step relative presentations]
\label{prop:one-step-relative-presentations}
Assume \textup{(C$_{\mathcal X}$)} and put
$G_{\mathcal X}(X,Y)=F_AX\oplus K_p$.  Every
$W\in\mathcal T_{X,Y}$ admits an epimorphism
\[
 d:G_0\twoheadrightarrow W
\]
with $G_0\in\addc G_{\mathcal X}(X,Y)$ and
$\ker d\in\mathcal T_{X,Y}$.
\end{proposition}

\begin{proof}
Write $W=(U,V,u,v)$.  Choose the standard evaluation right
approximations
\[
 \alpha:X^{\oplus m}\longrightarrow U,
 \qquad
 \beta:Y^{\oplus n}\longrightarrow V.
\]
By Lemma~\ref{lem:relative-covers}, they are epimorphisms and
$\ker\alpha\in\mathcal X$, $\ker\beta\in\mathcal Y$.  Under the two
adjunctions, they induce $\Lambda$-morphisms
$\widetilde\alpha:(F_AX)^{\oplus m}\to W$ and
$\widetilde\beta:(F_BY)^{\oplus n}\to W$.  The source of
$\widetilde\alpha$ already lies in $\addc G_{\mathcal X}(X,Y)$, whereas
$(F_BY)^{\oplus n}$ need not.  We therefore use the correction $a_p$.
Since $W\in\mathcal T_{X,Y}\subseteq
\mathcal C_{\mathcal X}$ and $a_p^{\oplus n}$ is a left
$\mathcal C_{\mathcal X}$-approximation, $\widetilde\beta$ factors as
\[
 (F_BY)^{\oplus n}\xrightarrow{a_p^{\oplus n}}K_p^{\oplus n}
 \xrightarrow{g}W.
\]
Write $g=(q,r)$ for its $e$- and $f$-components, and write $\lambda_n$
and $\mu_n$ for the direct sums of $\lambda$ and $\mu$.  Comparing the
$f$-components in $\widetilde\beta=g a_p^{\oplus n}$ gives
\[
 r\lambda_n=\beta.
\]
Moreover, since $g$ is a $\Lambda$-morphism, its two components satisfy
\[
 r\mu_n=u(q\otimes1_N).
\]
Set
\[
d=\begin{bmatrix}\widetilde\alpha&g\end{bmatrix}:
 (F_AX)^{\oplus m}\oplus K_p^{\oplus n}\longrightarrow W.
\]
Since $\widetilde\alpha$ is a $\Lambda$-morphism with $e$-component
$\alpha$, the morphism relation forces its $f$-component to be
$u(\alpha\otimes1_N)$.  Hence the corner components of $d$ are
\[
 d_e=\begin{bmatrix}\alpha&q\end{bmatrix},
 \qquad
 d_f=\begin{bmatrix}u(\alpha\otimes1_N)&r\end{bmatrix}.
\]
Here $d_e$ is epic because $\alpha$ is epic.  Moreover,
$\beta=r\lambda_n$ is epic, so $r$, and hence $d_f$, is epic.  Therefore
$d$ is epic.

Put $Z=X_Y^{\oplus n}$ and $Z_N=Z\otimes_AN$.  Since taking the
$e$-corner is exact, $(\ker d)e=\ker d_e$.
We use the following elementary kernel observation repeatedly: if
$a:L\twoheadrightarrow H$ and $b:Z'\to H$, then projection to $Z'$ gives
an exact sequence
\[
 0\longrightarrow\ker a\longrightarrow
 \ker\begin{bmatrix}a&b\end{bmatrix}
 \longrightarrow Z'\longrightarrow0.
\]
Indeed, for $z'\in Z'$ choose $x\in L$ with $a(x)=-b(z')$.
Applying this observation to $d_e=[\,\alpha\ q\,]$ gives
\[
 0\longrightarrow\ker\alpha\longrightarrow(\ker d)e
 \longrightarrow Z\longrightarrow0.
\]
Since $\ker\alpha,Z\in\mathcal X$, it follows that
$(\ker d)e\in\mathcal X$.

It remains to prove that $(\ker d)f\in\mathcal Y$.  Put
\[
 h=u(q\otimes1_N):Z_N\longrightarrow V,
 \qquad
 c=\begin{bmatrix}\lambda_n&\mu_n\end{bmatrix}:
 Y^{\oplus n}\oplus Z_N\twoheadrightarrow V_p^{\oplus n}.
\]
Here $c$ is the canonical epimorphism from the pushout, and the two morphism
identities above give
\[
 rc=\begin{bmatrix}\beta&h\end{bmatrix}.
\]
To control $\ker r$, set
\[
 E=\ker\begin{bmatrix}\beta&h\end{bmatrix}.
\]
Since $\beta$ is epic, the kernel observation gives
\[
 0\longrightarrow\ker\beta\longrightarrow E
 \longrightarrow Z_N\longrightarrow0.
\]
Both end terms lie in $\mathcal Y$: the first by construction and the second
by \eqref{eq:tensor-all-X}.  Hence $E\in\mathcal Y$.  Moreover, the
restriction $c|_E:E\twoheadrightarrow\ker r$ is epic.  Indeed, if
$\xi\in\ker r$, choose a preimage of $\xi$ under $c$; the equality
$rc=[\,\beta\ h\,]$ shows that this preimage lies in $E$.  Thus
$\ker r\in\mathcal Y$.

Finally, $(\ker d)f=\ker d_f$ and
\[
 d_f=\begin{bmatrix}u(\alpha\otimes1_N)&r\end{bmatrix}.
\]
Because $\beta=r\lambda_n$ is epic, $r$ is epic.  Applying the symmetric
form of the kernel observation, now using the second entry $r$, gives
\[
 0\longrightarrow\ker r\longrightarrow(\ker d)f
 \longrightarrow X^{\oplus m}\otimes_AN\longrightarrow0.
\]
The two end terms lie in $\mathcal Y$, again by the preceding paragraph and
\eqref{eq:tensor-all-X}.  Therefore $(\ker d)f\in\mathcal Y$, and hence
$\ker d\in\mathcal T_{X,Y}$.
\end{proof}

\begin{theorem}[One-step approximation gluing]\label{thm:main}
Let $X\in\stau A$ and $Y\in\stau B$ satisfy
\textup{(C$_{\mathcal X}$)}, and let $p$, $K_p$, and $a_p$ be as above.
Then $a_p:F_BY\to K_p$ is a minimal left
$\mathcal T_{X,Y}$-approximation, and
$G_{\mathcal X}(X,Y)=F_AX\oplus K_p$ is support $\tau$-tilting.  Moreover,
\[
 G_{\mathcal X}(X,Y)\cong P(\mathcal T_{X,Y}),
 \qquad
 \Fac_\Lambda G_{\mathcal X}(X,Y)=\mathcal T_{X,Y}.
\]
\end{theorem}

\begin{proof}
The observations preceding
Proposition~\ref{prop:one-step-relative-presentations} prove the assertion
about $a_p$ and show that the two correction ladders terminate at $F_AX$ and
$K_p$, respectively.  Hence
Theorem~\ref{thm:bilateral}\textup{(1)}--\textup{(2)}
gives
\[
 \Ext^1_\Lambda(G_{\mathcal X}(X,Y),\mathcal T_{X,Y})=0,
 \qquad
 \Fac_\Lambda G_{\mathcal X}(X,Y)=\mathcal T_{X,Y}.
\]
It also shows that $\mathcal T_{X,Y}$ is functorially finite.
Let $L$ be an indecomposable Ext-projective object of
$\mathcal T_{X,Y}$.  The epimorphism supplied by
Proposition~\ref{prop:one-step-relative-presentations} for $W=L$ has
kernel in $\mathcal T_{X,Y}$ and therefore splits.  Its source is a
direct summand of a finite direct sum of copies of
$G_{\mathcal X}(X,Y)$, so $L$ is isomorphic to an indecomposable summand
of $G_{\mathcal X}(X,Y)$.  Conversely, the displayed Ext-vanishing shows
that every indecomposable summand of $G_{\mathcal X}(X,Y)$ is
Ext-projective in $\mathcal T_{X,Y}$.  Thus the basic modules
$G_{\mathcal X}(X,Y)$ and $P(\mathcal T_{X,Y})$ have the same
indecomposable summands, and hence are isomorphic.
Theorem~\ref{thm:AIR-torsion}\textup{(2)} now shows that
$G_{\mathcal X}(X,Y)$ is support $\tau$-tilting.
\end{proof}

The next example shows concretely that the one-step correction can succeed
when direct corner induction does not, even for a nonzero connecting map.

\begin{example}[A nonzero connecting map]\label{ex:nonzero}
Let $Q$ be the quiver
\[
 1\xrightarrow{\alpha}3\xrightarrow{\beta}2
\]
and let $\Lambda=kQ$.  Set $e=e_1+e_2$ and $f=e_3$.  Then $\Lambda$ is
the Morita context algebra associated with
\[
 A=e\Lambda e\cong k(1\longrightarrow2),\qquad B=f\Lambda f\cong k,
 \qquad N=e\Lambda f=k\alpha,\qquad M=f\Lambda e=k\beta.
\]
Write $\mathbf 1,\mathbf 2$ for the simple $A$-modules at vertices $1,2$
and $\mathbf 3=B$; the same symbols denote the corresponding simple
$\Lambda$-modules.  We use vertical notation for radical series.
If $\gamma$ denotes the unique path from $1$ to $2$, multiplication gives
\[
 \phi(\alpha\otimes\beta)=\gamma\ne0,
 \qquad \psi=0.
\]
Thus the Morita context has a genuinely nonzero connecting map.

Take $X=\mathbf 1$ as an $A$-module and $Y=\mathbf 3=B$ as a
$B$-module.  Both are support $\tau$-tilting, and
\[
 X\otimes_AN\cong\mathbf 3=Y,
 \qquad
 Y\otimes_BM\cong\mathbf 2.
\]
Hence \textup{(C$_{\mathcal X}$)} holds.  Since
$\Hom_A(\mathbf 2,\mathbf 1)=0$, the zero map
\[
 p:\mathbf 2\longrightarrow0
\]
is the minimal left $\Fac_A\mathbf 1$-approximation.  Moreover,
$\mathbf 2\otimes_AN=0$, so the pushout defining $K_p$ gives
$K_p\cong\mathbf 3$.  Since
\[
 F_A\mathbf 1\cong
 \genfrac{}{}{0pt}{}{\mathbf 1}{\mathbf 3},
\]
Theorem~\ref{thm:main} yields the support $\tau$-tilting module
\[
 G_{\mathcal X}(\mathbf 1,\mathbf 3)
 \cong
 \genfrac{}{}{0pt}{}{\mathbf 1}{\mathbf 3}
 \oplus\mathbf 3.
\]

By contrast,
\[
 F_B\mathbf 3\cong
 \genfrac{}{}{0pt}{}{\mathbf 3}{\mathbf 2},
 \qquad
 F_A\mathbf 1\oplus F_B\mathbf 3
 \cong
 \genfrac{}{}{0pt}{}{\mathbf 1}{\mathbf 3}
 \oplus
 \genfrac{}{}{0pt}{}{\mathbf 3}{\mathbf 2}.
\]
This directly induced module has two indecomposable summands but is nonzero
at all three vertices.  Hence
$\Hom_\Lambda(e_i\Lambda,F_A\mathbf 1\oplus F_B\mathbf 3)\ne0$ for
$i=1,2,3$, so it has no nonzero projective complement.  Since
$|\Lambda|=3$, it cannot be support $\tau$-tilting.  Thus the correction
replaces the summand
$\genfrac{}{}{0pt}{}{\mathbf 3}{\mathbf 2}$ by $\mathbf 3$ and produces
a support $\tau$-tilting module.
\end{example}

The triangular specialization recovers Zhang's construction.

\begin{corollary}[{\cite[Theorem~5.1]{Zhang}}]\label{cor:Zhang}
Let
\[
 \Lambda=\begin{pmatrix}A&0\\M&B\end{pmatrix}
\]
be a triangular matrix algebra, and let $X\in\stau A$ and
$Y\in\stau B$.  If
\[
 p:Y\otimes_BM\longrightarrow X_Y
\]
is a minimal left $\Fac_AX$-approximation, then
\[
 (X,0)\oplus(X_Y,Y)_p
\]
is a support $\tau$-tilting $\Lambda$-module.
\end{corollary}

\begin{proof}
Here $N=0$, so condition \textup{(C$_{\mathcal X}$)} is automatic.  The pushout
\eqref{eq:pushout} reduces to $V_p=Y$, one has $F_AX=(X,0)$, and
$K_p=(X_Y,Y)_p$.  Thus Theorem~\ref{thm:main} is exactly the stated
construction.
\end{proof}

\subsection{The opposite construction}

We prepare the symmetric construction.  Let $X\in\stau A$ and
$Y\in\stau B$ and assume
\begin{equation*}\tag{C$_{\mathcal Y}$}
 Y\otimes_BM\in\mathcal X.
\end{equation*}
Since tensor products are right exact, it follows that
\[
 V\otimes_BM\in\mathcal X
 \qquad\text{for every }V\in\mathcal Y.
\]
Let
\[
 q:X\otimes_AN\longrightarrow Y_X
\]
be a minimal left $\mathcal Y$-approximation.  Form the pushout in $\modu A$
\[
\begin{tikzcd}[column sep=4.5em,row sep=3.0em]
 X\otimes_AN\otimes_BM
   \arrow[r,"1_X\otimes\phi"]
   \arrow[d,"q\otimes1_M"']
 &X\arrow[d,"\lambda'"]\\
 Y_X\otimes_BM\arrow[r,"\nu"']
 &U_q.
\end{tikzcd}
\]
The construction of Section~3.1 gives a unique $B$-module homomorphism
\[
 \overline q:U_q\otimes_AN\longrightarrow Y_X
\]
satisfying
\begin{align*}
 \overline q(\lambda'(x)\otimes n)&=q(x\otimes n),\\
 \overline q(\nu(y\otimes m)\otimes n)&=y\psi(m\otimes n).
\end{align*}
Set
\[
 K_q=(U_q,Y_X,\overline q,\nu),
 \qquad
 a_q=(\lambda',q):F_AX\longrightarrow K_q.
\]
By Proposition~\ref{prop:well-defined-corrections}, $K_q$ is a
$\Lambda$-module and $a_q$ is its canonical correction morphism.  Since
$Y_X\in\mathcal Y$, we have $Y_X\otimes_BM\in\mathcal X$.  Moreover,
$U_q$ is a quotient of $X\oplus(Y_X\otimes_BM)$, so $U_q\in\mathcal X$.
Hence $K_q\in\mathcal T_{X,Y}$.

By Lemma~\ref{lem:elementary-correction}, $a_q$ is a minimal left
$\mathcal C_{\mathcal Y}$-approximation.  Its target belongs to
$\mathcal T_{X,Y}\subseteq\mathcal C_{\mathcal Y}$, so it is also a
minimal left $\mathcal T_{X,Y}$-approximation.  Thus the $X$-ladder
terminates after at most one correction at $K_q$, while the $Y$-ladder has
already terminated at $F_BY$.

\begin{theorem}[Symmetric approximation gluing]\label{thm:symmetric}
With the notation above,
$G_{\mathcal Y}(X,Y)=K_q\oplus F_BY$ is support $\tau$-tilting, and
\[
 G_{\mathcal Y}(X,Y)\cong P(\mathcal T_{X,Y}),
 \qquad
 \Fac_\Lambda G_{\mathcal Y}(X,Y)=\mathcal T_{X,Y}.
\]
\end{theorem}

\begin{proof}
Apply Theorem~\ref{thm:main} after interchanging $A$ with $B$, $N$ with
$M$, and $\phi$ with $\psi$.
\end{proof}

\section{Direct corner induction}

Put $\mathbb I(X,Y)=F_AX\oplus F_BY$.  The direct construction is the
zero-correction case of the preceding one-sided theorem, which immediately
gives the following sufficient condition.

\begin{corollary}[Direct-induction criterion: sufficient direction]
\label{cor:direct-sufficient}
For an arbitrary Morita context algebra, let $X\in\stau A$ and
$Y\in\stau B$.  If
\[
 Y\otimes_BM\in\Fac_AX,
 \qquad
 X\otimes_AN\in\Fac_BY,
\]
then $\mathbb I(X,Y)$ is support $\tau$-tilting and
\[
 \Fac_\Lambda\mathbb I(X,Y)=\mathcal T_{X,Y}.
\]
\end{corollary}

\begin{proof}
In Theorem~\ref{thm:main}, take
$p=1_{Y\otimes_BM}$, which is a minimal left $\mathcal X$-approximation.
The universal property of \eqref{eq:pushout} identifies the resulting
pushout module with $K_p\cong F_BY$.
\end{proof}

For the converse, we impose the following condition on the connecting maps:
\begin{equation*}\tag{R}
 \operatorname{Im}\phi\subseteq J(A),
 \qquad
 \operatorname{Im}\psi\subseteq J(B).
\end{equation*}
By the multiplication formula in Section~2.1, this is equivalently
\[
 NM\subseteq J(A),\qquad MN\subseteq J(B).
\]
The hypothesis is automatic in the usual basic setting.  Indeed, if
$\Lambda$ is basic, then $\Lambda/J(\Lambda)$ is a product of division
algebras.  The images of the two diagonal idempotents $e$ and $f=1-e$ are
therefore complementary central idempotents, so
$e\Lambda f=N$ and $f\Lambda e=M$ lie in $J(\Lambda)$.  The corner identity
$J(e\Lambda e)=eJ(\Lambda)e$, together with the multiplication formula,
then gives \textup{(R)}.  In particular, it holds for a bound quiver algebra
and any partition of its vertex idempotents; this is the basic case of the
corner construction in \cite[Example~2.10]{GreenPsaroudakis}.

Conversely, when $A$ and $B$ are basic, the radical formula in
Lemma~\ref{lem:radical-context}\textup{(2)} shows that \textup{(R)} is
equivalent to $\Lambda$ being basic.  Thus basicness of $A$ and $B$ alone
does not suffice: the strict context $M_2(k)$ discussed below gives a concrete
counterexample.

\begin{theorem}[Converse and characterization for direct induction]
\label{thm:direct-characterization}
Assume \textup{(R)}, and let $X\in\modu A$ and $Y\in\modu B$.
Then the following are equivalent:
\begin{enumerate}[label=\textup{(\arabic*)},leftmargin=2.2em]
 \item $\mathbb I(X,Y)$ is a support $\tau$-tilting $\Lambda$-module;
 \item $X\in\stau A$, $Y\in\stau B$, and
 \[
  Y\otimes_BM\in\Fac_AX,
  \qquad
  X\otimes_AN\in\Fac_BY.
 \]
\end{enumerate}
When these equivalent conditions hold, let $(X,P_X)$ and $(Y,P_Y)$ be the
corresponding support $\tau$-tilting pairs.  Then
\[
 \bigl(\mathbb I(X,Y),F_AP_X\oplus F_BP_Y\bigr)
\]
is the corresponding support $\tau$-tilting pair over $\Lambda$.
Moreover,
\[
 \Fac_\Lambda\mathbb I(X,Y)=\mathcal T_{X,Y}.
\]
\end{theorem}

The converse uses the radical hypothesis to separate the indecomposable
summands induced from the two corners.  Preservation of minimal
projective presentations does not require this hypothesis.

\subsection{The converse under the radical condition}

We prove Theorem~\ref{thm:direct-characterization}.  We use the standard
facts collected in Section~2.2.  The radical formula below is classical;
see \cite{Sands}.  A short proof is included to make the role of
\textup{(R)} explicit.

\begin{lemma}[Corner induction and the radical formula]\label{lem:radical-context}
\begin{enumerate}[label=\textup{(\arabic*)},leftmargin=2.2em]
 \item Without assumptions on $\phi$ or $\psi$, the functors $F_A$ and $F_B$
 preserve projective covers and minimal projective presentations.
 \item If \textup{(R)} holds, then
 \begin{equation}\label{eq:radical-formula}
  J(\Lambda)=
  \begin{pmatrix}J(A)&N\\M&J(B)\end{pmatrix},
  \qquad
 \Lambda/J(\Lambda)\cong A/J(A)\times B/J(B).
 \end{equation}
 Consequently, the simple $\Lambda$-modules are precisely $(S,0)$ and
 $(0,T)$, where $S$ and $T$ run through the simple $A$- and $B$-modules,
 respectively.  Hence $|\Lambda|=|A|+|B|$, and the indecomposable
 projective $\Lambda$-modules are precisely the modules $F_AP$ and $F_BQ$
 induced from indecomposable projective modules $P$ over $A$ and $Q$ over
 $B$.
\end{enumerate}
\end{lemma}

\begin{proof}
\textup{(1)} The functor $F_A$ is right exact.  Since it is left adjoint to
the exact restriction functor $(-)e$, it sends projective modules to
projective modules.  It is also fully faithful, and therefore preserves right
minimal morphisms: if $d:P\to L$ is right minimal and an endomorphism $h$
of $F_AP$ satisfies $(F_Ad)h=F_Ad$, full faithfulness gives $h=F_Ah_0$
and $dh_0=d$.  Thus $h_0$, and hence $h$, is invertible.

Let $P_1\xrightarrow{\sigma}P_0\xrightarrow{\pi}U\to0$ be a minimal
projective presentation.  Both $\pi$ and $\sigma$ are right minimal.
Right exactness gives
\[
 F_AP_1\xrightarrow{F_A\sigma}F_AP_0
 \xrightarrow{F_A\pi}F_AU\longrightarrow0.
\]
The epimorphism $F_A\pi$ is right minimal, so it is a projective cover.
Set $K=\ker(F_A\pi)$.  Right exactness also gives
$\operatorname{Im}(F_A\sigma)=K$, so $F_A\sigma$ factors as
\[
 F_AP_1\xrightarrow{\rho}K\xrightarrow{\iota}F_AP_0,
\]
where $\rho$ is an epimorphism and $\iota$ is the kernel inclusion.  If an
endomorphism $h$ of $F_AP_1$ satisfies $\rho h=\rho$, then
\[
 (F_A\sigma)h=\iota\rho h=\iota\rho=F_A\sigma.
\]
Since $F_A\sigma$ is right minimal, $h$ is invertible.  Thus $\rho$ is
right minimal and hence is a projective cover of $K$.
This proves the assertion for $F_A$; the proof for $F_B$ is identical.

\textup{(2)}
Set
\[
 \mathfrak r=\begin{pmatrix}J(A)&N\\M&J(B)\end{pmatrix}.
\]
Condition~\textup{(R)} ensures that $\mathfrak r$ is a two-sided ideal:
the products of two off-diagonal entries lie in $J(A)$ or $J(B)$.
Moreover,
\[
 \Lambda/\mathfrak r\cong A/J(A)\times B/J(B)
\]
is semisimple, and therefore $J(\Lambda)\subseteq\mathfrak r$.

For the reverse inclusion, fix $z\in\mathfrak r$.  By the Jacobson radical
criterion, it is enough to prove that $1-\lambda z$ is invertible for every
$\lambda\in\Lambda$.  Since $\mathfrak r$ is an ideal, one has
$\lambda z\in\mathfrak r$.  Thus it suffices to show that $1-w$ is invertible
for every $w\in\mathfrak r$.  Write
\[
 w=\begin{pmatrix}a&n\\m&b\end{pmatrix}\in\mathfrak r
\]
and put $p=1_A-a$ and $q=1_B-b$.  Both $p$ and $q$ are invertible.  By
\textup{(R)},
\[
 s=q-\psi(m\otimes p^{-1}n)\in1_B+J(B),
\]
so $s$ is also invertible.  Direct multiplication gives
\[
 1-w=\begin{pmatrix}p&-n\\-m&q\end{pmatrix}
 =
 \begin{pmatrix}1&0\\-mp^{-1}&1\end{pmatrix}
 \begin{pmatrix}p&0\\0&s\end{pmatrix}
 \begin{pmatrix}1&-p^{-1}n\\0&1\end{pmatrix},
\]
and all three factors are invertible.  Hence $1-w$ is invertible for every
$w\in\mathfrak r$.  Applying this to $w=\lambda z$ shows that
$1-\lambda z$ is invertible for every $\lambda\in\Lambda$.  Therefore
$z\in J(\Lambda)$.  Since $z$ was arbitrary,
$\mathfrak r\subseteq J(\Lambda)$, and \eqref{eq:radical-formula} follows.

It remains to deduce the assertions about simple and projective modules.
The semisimple quotient in \eqref{eq:radical-formula} shows that every
simple $\Lambda$-module is uniquely of the form $(S,0)$ or $(0,T)$, with
$S$ simple over $A$ and $T$ simple over $B$.  To compute the tops of
induced modules, recall that
\[
 \operatorname{top}_\Lambda L
 \cong L\otimes_\Lambda\Lambda/J(\Lambda).
\]
Since \eqref{eq:radical-formula} gives
$e(\Lambda/J(\Lambda))\cong(A/J(A),0)$, for $U\in\modu A$ we obtain
\[
\begin{aligned}
 \operatorname{top}_\Lambda(F_AU)
 &\cong (U\otimes_Ae\Lambda)\otimes_\Lambda\Lambda/J(\Lambda)\\
 &\cong U\otimes_Ae(\Lambda/J(\Lambda))
 \cong(\operatorname{top}_AU,0).
\end{aligned}
\]
Thus the second component disappears because the off-diagonal bimodules
vanish in the semisimple quotient.  Symmetrically, for $V\in\modu B$,
\[
 \operatorname{top}_\Lambda(F_BV)
 \cong(0,\operatorname{top}_BV).
\]
By part~\textup{(1)}, induction preserves projective
modules.  Hence, if $P$ and $Q$ are indecomposable projective modules over
$A$ and $B$, respectively, then $F_AP$ and $F_BQ$ are indecomposable
projective modules with simple tops supported at different corners.
Conversely, every indecomposable projective $\Lambda$-module is the
projective cover of one of the simple modules above.  This proves both the
classification and the equality $|\Lambda|=|A|+|B|$.
\end{proof}

Under \textup{(R)}, the same argument gives, for $X,Y$ and for projectives
$P,Q$ having no repeated isomorphic indecomposable summands,
\begin{equation}\label{eq:summand-counts}
 |F_AX\oplus F_BY|=|X|+|Y|,
 \qquad
 |F_AP\oplus F_BQ|=|P|+|Q|.
\end{equation}
Indeed, full faithfulness preserves indecomposable summands within each
corner, while \eqref{eq:radical-formula} separates the supports of their
tops.

\begin{lemma}[$\tau$-rigidity test]\label{lem:tau-rigid-direct}
For arbitrary connecting maps and $X\in\modu A$, $Y\in\modu B$, the
module $\mathbb I(X,Y)$ is $\tau$-rigid if and only if $X$ and $Y$ are
$\tau$-rigid and
\begin{equation}\label{eq:cross-Hom}
 \Hom_A(Y\otimes_BM,\tau_AX)=0,
 \qquad
 \Hom_B(X\otimes_AN,\tau_BY)=0.
\end{equation}
\end{lemma}

\begin{proof}
Put $Z=F_AX\oplus F_BY$, with the actual multiplicities of its summands;
deleting repeated isomorphic summands does not affect $\tau$-rigidity.
Choose minimal projective presentations
\[
 \sigma_X:P_1^X\longrightarrow P_0^X,
 \qquad
 \sigma_Y:P_1^Y\longrightarrow P_0^Y.
\]
By Lemma~\ref{lem:radical-context}\textup{(1)},
\[
 \sigma_Z:=F_A\sigma_X\oplus F_B\sigma_Y
\]
is a minimal projective presentation of $Z$.  By
Proposition~\ref{prop:AIR-criteria}\textup{(1)}, $Z$ is $\tau$-rigid
if and only if $\Hom_\Lambda(\sigma_Z,Z)$ is surjective.

Both $\sigma_Z$ and $Z$ have one summand induced from each corner.
Consequently, $\Hom_\Lambda(\sigma_Z,Z)$ is the direct sum of four maps.
Using the adjunctions $F_A\dashv(-)e$ and $F_B\dashv(-)f$, these maps are
identified as follows:
\[
\begin{aligned}
 \Hom_\Lambda(F_A\sigma_X,F_AX)
   &\cong \Hom_A(\sigma_X,X),\\
 \Hom_\Lambda(F_A\sigma_X,F_BY)
   &\cong \Hom_A(\sigma_X,Y\otimes_BM),\\
 \Hom_\Lambda(F_B\sigma_Y,F_AX)
   &\cong \Hom_B(\sigma_Y,X\otimes_AN),\\
 \Hom_\Lambda(F_B\sigma_Y,F_BY)
   &\cong \Hom_B(\sigma_Y,Y).
\end{aligned}
\]
Hence $\Hom_\Lambda(\sigma_Z,Z)$ is surjective exactly when all four maps
on the right are surjective.

For a minimal projective presentation $\sigma_T$,
Proposition~\ref{prop:AIR-criteria}\textup{(1)} gives
\[
 \Hom_C(\sigma_T,L)\text{ surjective}
 \quad\Longleftrightarrow\quad
 \Hom_C(L,\tau_CT)=0.
\]
Applied to the first and fourth maps, this says precisely that $X$ and
$Y$ are $\tau$-rigid.  Applied to the second and third maps, it gives,
respectively,
\[
 \Hom_A(Y\otimes_BM,\tau_AX)=0,
 \qquad
 \Hom_B(X\otimes_AN,\tau_BY)=0.
\]
These are exactly the two conditions in \eqref{eq:cross-Hom}.
\end{proof}

\begin{proof}[Proof of Theorem~\ref{thm:direct-characterization}]
Suppose first that $Z=\mathbb I(X,Y)$ is support $\tau$-tilting.  By
Lemma~\ref{lem:tau-rigid-direct}, $X$ and $Y$ are $\tau$-rigid and
\eqref{eq:cross-Hom} holds.  Let $Q_X$ be the direct sum of one representative
from each isomorphism class of indecomposable projective $A$-modules $Q$
satisfying $\Hom_A(Q,X)=0$, and
define $Q_Y$ symmetrically.  The $\tau$-rigid-pair inequalities give
\[
 |X|+|Q_X|\leq|A|,
 \qquad
 |Y|+|Q_Y|\leq|B|.
\]
Here we use the standard bound for $\tau$-rigid pairs from
Proposition~\ref{prop:AIR-criteria}\textup{(2)}.

Let $P_Z$ be the projective complement of $Z$.  By
Lemma~\ref{lem:radical-context}, every indecomposable summand of $P_Z$ is
induced from one corner.  If $F_AQ$ is such a summand, then adjunction gives
\[
 0=\Hom_\Lambda(F_AQ,Z)
  \cong\Hom_A\bigl(Q,X\oplus(Y\otimes_BM)\bigr),
\]
so $Q$ is a summand of $Q_X$; the $B$-side is symmetric.  Hence
\[
 P_Z\in\addc(F_AQ_X\oplus F_BQ_Y).
\]
Using \eqref{eq:summand-counts} and the preceding inequalities, we obtain
\[
 |A|+|B|=|\Lambda|=|Z|+|P_Z|
 \leq |X|+|Q_X|+|Y|+|Q_Y|
 \leq |A|+|B|.
\]
All inequalities are equalities.  Thus $(X,Q_X)$ and $(Y,Q_Y)$ are support
$\tau$-tilting pairs.  Every indecomposable summand of $P_Z$ occurs in
$F_AQ_X\oplus F_BQ_Y$, and the two basic modules have the same number of
summands; hence $P_Z\cong F_AQ_X\oplus F_BQ_Y$.
Orthogonality of $P_Z$ to $Z$ gives
\[
 \Hom_A(Q_X,Y\otimes_BM)=0,
 \qquad
 \Hom_B(Q_Y,X\otimes_AN)=0.
\]
Together with \eqref{eq:cross-Hom} and
\eqref{eq:Fac-orthogonality}, this proves the two cross-component
membership conditions.  Moreover, $Q_X=P_X$ and $Q_Y=P_Y$, so the same
argument gives
\[
 P_Z\cong F_AP_X\oplus F_BP_Y,
\]
which proves the assertion about the projective complement.

Conversely, suppose that $X\in\stau A$, $Y\in\stau B$, and
\[
 Y\otimes_BM\in\Fac_AX,
 \qquad
 X\otimes_AN\in\Fac_BY.
\]
Corollary~\ref{cor:direct-sufficient} then implies that $\mathbb I(X,Y)$
is support $\tau$-tilting and
\[
 \Fac_\Lambda\mathbb I(X,Y)=\mathcal T_{X,Y}.
\]
This proves the converse; neither conclusion uses \textup{(R)}.
\end{proof}

The radical hypothesis cannot be omitted from this converse.  For the
strict context $\Lambda=M_2(k)$ with $A=B=M=N=k$, take $X=k$ and $Y=0$.
Here $A$ and $B$ are basic, but \textup{(R)} fails and $\Lambda$ is not
basic.  Although the diagonal idempotents $e$ and $f$ are primitive and
orthogonal, the projectives $e\Lambda$ and $f\Lambda$ are isomorphic; hence
they represent only one isomorphism class of indecomposable projectives.
Equivalently, $M_2(k)$ has only one isomorphism class of simple modules.
Now $F_AX=e\Lambda$ is the unique indecomposable projective up to
isomorphism, so $\mathbb I(X,Y)$ is $\tau$-tilting.  However,
$X\otimes_AN=k\notin\Fac_B0$.  In this example the projective modules
induced from the two corners are isomorphic, so the summand count fails.

\subsection{Consequences and specializations}

\begin{corollary}\label{cor:triangular-direct}
Let
\[
 \Lambda=\begin{pmatrix}A&0\\M&B\end{pmatrix},
\]
and let $(X,P_X)$ and $(Y,P_Y)$ be support $\tau$-tilting pairs over $A$
and $B$.  Then
\[
 (X,0)\oplus(Y\otimes_BM,Y)
\]
is support $\tau$-tilting over $\Lambda$ if and only if
\[
 Y\otimes_BM\in\Fac_AX.
\]
By \eqref{eq:Fac-orthogonality}, this is equivalent to
\[
 \Hom_A(Y\otimes_BM,\tau_AX)=0,
 \qquad
 \Hom_A(P_X,Y\otimes_BM)=0.
\]
\end{corollary}

\begin{proof}
Here $N=0$, so $X\otimes_AN=0\in\Fac_BY$; hence
Theorem~\ref{thm:direct-characterization} leaves exactly the single condition
$Y\otimes_BM\in\Fac_AX$.  Formula~\eqref{eq:Fac-orthogonality} gives its two
equivalent Hom-vanishing conditions.  The radical condition is automatic,
and the argument is valid over an arbitrary field.  Over an algebraically
closed field, the Hom-vanishing formulation is the criterion of Gao and
Huang \cite[Theorem~4.3]{GaoHuang}.  Thus the present result removes their
assumption that the ground field be algebraically closed.
\end{proof}

The next result gives a condition-free construction when the corner algebras
are $\tau$-tilting finite.  It necessarily allows the corner torsion classes
to grow.  For a class $\mathcal S\subseteq\modu C$, write
$\operatorname{Tors}_C(\mathcal S)$ for the smallest torsion class containing
$\mathcal S$.

\begin{proposition}[Least compatible enlargement]
\label{prop:finite-closure}
Assume that $A$ and $B$ are $\tau$-tilting finite.  Starting with
$\mathcal X_0=\mathcal X$ and $\mathcal Y_0=\mathcal Y$, define
\begin{align*}
 \mathcal X_{i+1}
 &=\operatorname{Tors}_A\bigl(
     \mathcal X_i\cup
     \{V\otimes_BM\mid V\in\mathcal Y_i\}\bigr),\\
 \mathcal Y_{i+1}
 &=\operatorname{Tors}_B\bigl(
     \mathcal Y_i\cup
     \{U\otimes_AN\mid U\in\mathcal X_i\}\bigr).
\end{align*}
The two ascending sequences stabilize.  Their stable values
$(\widehat{\mathcal X},\widehat{\mathcal Y})$ form the least pair of torsion
classes containing $(\mathcal X,\mathcal Y)$ and satisfying
\[
 \{V\otimes_BM\mid V\in\widehat{\mathcal Y}\}
 \subseteq\widehat{\mathcal X},
 \qquad
 \{U\otimes_AN\mid U\in\widehat{\mathcal X}\}
 \subseteq\widehat{\mathcal Y}.
\]
If $\widehat X$ and $\widehat Y$ are the support $\tau$-tilting
generators of these stable torsion classes, then
\[
 F_A\widehat X\oplus F_B\widehat Y
\]
is a support $\tau$-tilting $\Lambda$-module and its torsion class is
$\mathcal T_{\widehat X,\widehat Y}$.  No condition on $\phi$ or $\psi$ is
required.
\end{proposition}

\begin{proof}
By construction, each $\mathcal X_i$ and $\mathcal Y_i$ is a torsion class,
and
\[
 \mathcal X_i\subseteq\mathcal X_{i+1},
 \qquad
 \mathcal Y_i\subseteq\mathcal Y_{i+1}.
\]
By \cite[Theorem~2.7]{AIR} and \cite[Theorem~3.8]{DIJ}, a $\tau$-tilting
finite algebra has only finitely many torsion classes, all of which are
functorially finite.  Hence there is an $r\geq0$ such that
\[
 \mathcal X_r=\mathcal X_{r+1},
 \qquad
 \mathcal Y_r=\mathcal Y_{r+1}.
\]
Set $\widehat{\mathcal X}=\mathcal X_r$ and
$\widehat{\mathcal Y}=\mathcal Y_r$.  The defining equalities at this stable
stage give
\[
 \{V\otimes_BM\mid V\in\widehat{\mathcal Y}\}
 \subseteq\widehat{\mathcal X},
 \qquad
 \{U\otimes_AN\mid U\in\widehat{\mathcal X}\}
 \subseteq\widehat{\mathcal Y}.
\]

To prove minimality, let $(\mathcal U,\mathcal V)$ be any pair of torsion
classes containing $(\mathcal X,\mathcal Y)$ and satisfying
\[
 \{W\otimes_BM\mid W\in\mathcal V\}\subseteq\mathcal U,
 \qquad
 \{L\otimes_AN\mid L\in\mathcal U\}\subseteq\mathcal V.
\]
We prove by induction that
$\mathcal X_i\subseteq\mathcal U$ and
$\mathcal Y_i\subseteq\mathcal V$ for every $i$.  This is clear for $i=0$.
If it holds for $i$, then
\[
 \mathcal X_i\cup
 \{W\otimes_BM\mid W\in\mathcal Y_i\}
 \subseteq\mathcal U.
\]
Since $\mathcal U$ is a torsion class, the definition of $\mathcal X_{i+1}$
gives $\mathcal X_{i+1}\subseteq\mathcal U$.  The same argument gives
$\mathcal Y_{i+1}\subseteq\mathcal V$.  Thus
\[
 \widehat{\mathcal X}\subseteq\mathcal U,
 \qquad
 \widehat{\mathcal Y}\subseteq\mathcal V,
\]
which proves minimality.

Finally, the stable classes are functorially finite, so let $\widehat X$ and
$\widehat Y$ be their support $\tau$-tilting generators.  Thus
\[
 \widehat{\mathcal X}=\Fac_A\widehat X,
 \qquad
 \widehat{\mathcal Y}=\Fac_B\widehat Y.
\]
The compatibility inclusions yield
\[
 \widehat Y\otimes_BM\in\Fac_A\widehat X,
 \qquad
 \widehat X\otimes_AN\in\Fac_B\widehat Y.
\]
Corollary~\ref{cor:direct-sufficient} therefore shows that
$F_A\widehat X\oplus F_B\widehat Y$ is support $\tau$-tilting and that its
torsion class is $\mathcal T_{\widehat X,\widehat Y}$.  Since that corollary
imposes no condition on $\phi$ or $\psi$, neither does this conclusion.
\end{proof}

\section*{Acknowledgements}
The author gratefully acknowledges the assistance of ChatGPT in the preparation
and revision of this manuscript.  This work was supported by the National
Natural Science Foundation of China (Grant Nos.~12201211 and 126701258).

\enlargethispage{2\baselineskip}
\end{document}